\documentclass{amsart}
\usepackage{eurosym}
\usepackage{amsthm,amssymb,amsfonts,amsmath,mathrsfs}
\usepackage{dsfont}
\usepackage{graphicx}
\usepackage{enumerate}
\usepackage{hyperref}
\usepackage[nodvipsnames]{xcolor}
\usepackage{soul}

\newcommand{\notaremm}[1]{\textcolor{blue}{\ifmmode\text{\sout{\ensuremath{#1}}}
\else\sout{#1}\fi }}

\newtheorem{theorem}{Theorem}

\newtheorem{corollary}[theorem]{Corollary}

\newtheorem{definition}[theorem]{Definition}

\newtheorem{lemma}[theorem]{Lemma}

\newtheorem{proposition}[theorem]{Proposition}
\newtheorem{remark}[theorem]{Remark}

\renewcommand{\S}{\mathbb{S}}

\subjclass[2010]{Primary 37H99 ; Secondary 37A30, 37C30 }
\keywords{Linear response, random dynamical system, Markov operator}
\input{tcilatex}

\begin{document}
\title[Quadratic response of random and deterministic dynamical systems]{%
Quadratic response of random and deterministic dynamical systems.}
\author{Stefano Galatolo}
\address{Dipartimento di Matematica, Universit\`a di Pisa, Largo Bruno
Pontecorvo 5, 56127 Pisa, Italy}
\email{stefano.galatolo@unipi.it}
\urladdr{http://pagine.dm.unipi.it/~a080288/}
\author{Julien Sedro}
\address{Laboratoire de Probabilit\'es, Statistique et Mod\'elisation
(LPSM), Sorbonne Universit\'e, Universit\'e de Paris, Place Jussieu 4, 75005
Paris, France.}
\email{sedro@lpsm.paris}
\date{\today }

\begin{abstract}
We consider the linear and quadratic higher order terms associated to the
response of the statistical properties of a dynamical system to suitable
small perturbations. These terms are related to the first and second
derivative of the stationary measure with respect to the changes in the
system itself, expressing how the statistical properties of the system
varies under the perturbation.\newline
We show a general framework in which one can obtain rigorous convergence and
formulas for these two terms. The framework is flexible enough to be applied
both to deterministic and random systems. We give examples of such an
application computing linear and quadratic response for Arnold maps with
additive noise and deterministic expanding maps.
\end{abstract}

\maketitle


\textbf{The statistical properties of the long time behavior of the
evolution of dynamical system are strongly related to the properties of its
invariant or stationary measures. It is important both in the theory and in
the applications to understand quantitatively how the invariant measures of
interest change when a given system is perturbed in some way. \ In the case
where the invariant measure changes smoothly with the perturbation, the
Linear and Quadratic Response express the first and second order leading
terms describing the change in the invariant measure with respect to the
perturbation, hence this concept is related to the first and second
derivative representing how the invariant measure change. The paper gives a
general approach for the understanding of these concepts in families of
Markov operators with suitable properties, which hold for natural
perturbations of transfer operators associated to deterministic and random
systems. We show quite general assumptions under which Linear and Quadratic
Response hold in these systems and explicit formulas to compute it. We show
applications both to deterministic and random systems, providing a unified
approach to these cases. As far as we know, formulas for quadratic response
terms in the random case are shown in this paper for the first time.}

\section{Introduction}

\noindent \textbf{Linear Response in the dynamical systems context.} Random
and deterministic dynamical systems are often used as models of physical or
social complex systems. In many cases it is natural to describe some aspects
of the evolution of a system having many components at different time and
size scales as a random input while other components evolve
deterministically.\footnote{%
Typically this is done by modeling the evolution of the system at a small
scale as a random\ perturbation of the large scale dynamics or, in the
presence of different time scales (fast-slow systems) one can model the
evolution of the fast component as a random perturbation of the slow one.
Sometimes random dynamical system appear as a model for an "infinite
dimensional limit" of deterministic dynamical systems having many
interacting components (for an example related to linear response see \cite%
{WG}).} \ For random dynamical systems, like in deterministic ones, the
invariant or stationary measures play a central role in the understanding of
the statistical properties of the system. It is then natural to
study the robustness of those invariant measures to perturbations of the
system, whether in its deterministic or random part. Another motivation comes when the system
of interest is submitted to a certain change or perturbation (an external
forcing e.g.): it is useful to understand and predict the direction and the
intensity of change of the invariant measures of the system, as it provides
information on the direction and the intensity of change of its statistical
properties after the perturbation.

When such a change is smooth, we say that the system exhibits Linear Response, and this can be
described by a suitable derivative. More precisely, but still informally, \
let $(S_{t})_{t\geq 0}$ be a one parameter family of dynamical systems
obtained by perturbing an initial system $S_{0}$, and let $h_{t}$ be\ the
invariant measure of interest of the systems $\ S_{t}$.
\\The linear response of $S_{0}$ under the given perturbation is defined by
the limit%
\begin{equation*}
R:=\lim_{t\rightarrow 0}\frac{h_{t}-h_{0}}{t}
\end{equation*}%
where the meaning of this convergence can vary from system to system. In
some systems and for a given perturbation, one may get $L^{1}$-convergence for this limit; in other systems or for other perturbations one may get convergence in
weaker or stronger topologies. The linear response to the
perturbation hence represents the first order term of the response of a
system to a perturbation and when it holds, a linear response formula can be
written:%
\begin{equation}
h_{t}=h_{0}+Rt+o(t)  \label{lin}
\end{equation}%
which holds in some weaker or stronger sense.

For deterministic dynamical systems, Linear Response, as well as higher-order formulae have been
obtained first by Ruelle, in the uniformly hyperbolic case \cite{R,R2}. Nowadays these
results have been extended to many other situations where one has some
hyperbolicity and enough smoothness for the system and its perturbations. On
the other hand there are many examples of deterministic systems whose
statistical properties do not behave smoothly under quite natural
perturbations. We refer to the survey \cite{BB} for an extended discussion
of the literature about linear response for deterministic systems. Since in
our paper we mainly consider the response in the random case, in the next
paragraphs we give more details on the literature for random systems, about which no surveys have been written until now.

\noindent \textbf{Linear Response for random dynamical systems.} In the
physical literature, often borrowing the point of view of statistical
mechanics, linear response formulae for several kinds of stochastic systems
and for several aspects of their statistical behavior have been proposed and
applied in various contexts (see \cite{HT} and \cite{bm} for general
surveys), notably in climate science where several applications and
estimation methods have been proposed, often in relation with the
understanding of the nature of tipping points in the climate evolution (see
the introduction of \cite{HM} or \cite{Lu2}, \cite{Lu1}, \cite{Lu3}, \cite%
{Lu4}). The mathematical literature about Linear Response for random dynamical systems
is smaller and more recent. In the next paragraphs, we try to give a quick overview of the possible approaches and existing results. As usual, there are mostly two perspectives to study random systems : the \emph{annealed} and \emph{quenched} point of view. Although we will focus on the former in this paper, we give a brief account of the latter at the end of this paragraph.
\\ Consider a random dynamical system, in which the future dynamics depend on the initial condition and on some evolution laws containing random parameters. In the annealed case, the focus is on the average behavior of the evolution of the system with respect to these random parameters. In this perspective, when the randomness is strong enough, statistical stability and linear response to perturbations could be expected or considered easier to establish compared to the deterministic case. 
However it is worth  remarking that even in this situation there are examples of non-smooth
statistical stability under natural perturbations. The skew
products given in \cite{Gpre} can be seen as random i.i.d rotations with a
polynomial speed of mixing, having only H\"older statistical stability under small
perturbations, even for very smooth observables. 
Regarding positive results, examples of linear response for small random perturbations of
deterministic systems appear in \cite{Li2} and \cite{GL}. In the paper \cite{ZH07} the smoothness of the
invariant measure response under suitable perturbations is proved for a
class of random diffeomorphisms, but no explicit formula is given for the derivatives. An application to the smoothness of the rotation number of Arnold circle maps with additive noise is presented. In the paper \cite{MSGDG}, these
findings are extended outside the diffeomorphism case and applied to an
idealized model of El Ni\~{n}o-Southern Oscillation. Linear response
results for suitable classes of random systems were proved in \cite{HM} {where} {the} technical
framework {was} adapted to stochastic differential equations and in \cite{BRS}, 
where the authors consider random compositions of expanding or
non-uniformly expanding maps. In the paper \cite{GG}, like in the present
paper in Section \ref{secnoise}, general discrete time systems with
additive noise are considered, i.e. systems where the dynamics map a point
deterministically to another point and then some random perturbation is
added independently at each iteration according to a certain noise distribution kernel. 
The response of the stationary measure to perturbations of the deterministic part of the system or to perturbations
of the shape of the noise is considered and explicit formulas for the
response are given, with convergence in different stronger or weaker spaces
according to the kind of perturbation considered. It is notable that
in the case of additive noise (like in the case considered by \cite{HM}) no
strong assumptions on the deterministic part of the dynamics are necessary,
and in particular no hyperbolicity assumptions are required. 
In some sense, in this approach the regularizing effect of the noise on suitable functional
spaces plays the role of the Lasota-Yorke-Doeblin-Fortet inequalities, as
commonly used in many other functional analytic approaches to the study of
the statistical properties of systems.
\\Another possible perspective on the
response of statistical properties to perturbations in random systems
concerns \emph{quenched result}, i.e when one looks at a fixed realization
of the random parameters instead of averaging over all possible values of the random
parameters. In this approach, one considers random products of maps instead of iterations of a single system, and the dynamically relevant objects become the Oseledets-Lyapunov spectrum of some transfer operator cocycle. One then studies the response of an appropriate
equivariant family of measure to the perturbations. Contrary to the annealed case, one does not expect some regularization effect to come into play. The interest of this
perspective was highlighted in the climate literature (notably \cite{CSG}), but
so far the mathematical results in this direction are very sparse: see \cite[Chapter 5]{Sedro}.

\noindent \textbf{Linear request, optimal response, numerical
methods. }An important problem related to linear response is the
control of the statistical properties of a system: \emph{how can one perturb
the system, in order to modify its statistical properties in a prescribed
way? How can one do it optimally?}\ \emph{What is the best action to be
taken in a possible set of allowed small perturbations in order to achieve a
wanted small modification of the statistical behavior of the system?} \ Understanding of this problem has a potentially great importance in the
applications of Linear Response, as it is related to questions about optimal
strategies in order to influence the behavior of a system. This problem was
considered from a mathematical point of view for deterministic systems in 
\cite{GP} and  \cite{Kl}. \ Similar problems in the case of extended systems were considered in \cite{Mac}. For random systems with
additive noise the problem was briefly considered in \cite{GG}. In \cite{ADF} the
problem was considered for systems described by finite states Markov chains.
Rigorous numerical approaches for the computation of the linear response are available to some extent, both for deterministic and random
systems (see \cite{BGNN, PV}). We note that in principle, the quadratic response  can provide important information in these optimization problems,
as it can be of help in establishing convexity properties in the response of
the statistical properties of a given family of systems under perturbation.

\noindent \textbf{Quadratic response and the present paper.} In the random case like in the deterministic
case, a fruitful strategy to study the
stability of a system relies on the remark that the stationary or invariant
measures of interest are fixed points of the transfer operator associated
to the system we consider; thus, linear response statements or quantitative
stability results can be proved by first proving perturbation theorems for
suitable operators, as done in \cite{HM, notes, JS, Gpre, KL, Li2,GG}.%
\\In this paper we adopt this point of view, proving two general
theorems about linear response of fixed points of Markov operators to
perturbations. Those statements are tailored for operators which naturally
appear as transfer operators of random or deterministic dynamical systems.

Once the first order (the linear part) of the response of a system
to a perturbation is understood, it is natural to study further orders. The second order of the response
may then be related to the second derivative and to other natural questions, as
convexity aspects of the response of the system under perturbation, or the
stability of the first order response. Hence, if the Linear Response 
$R,$ represents the first order term of the response (see \eqref{lin}), the Quadratic Response $Q$ will
represent the second order term of this response, analogous to the second
derivatives in usual Taylor's expansion:%
\begin{equation}
h_{t}=h_{0}+Rt+\frac{1}{2}Qt^{2}+o(t^{2}).  \label{Quad}
\end{equation}%

The first appearance of higher-order response formulae was in \cite{R}. In the existing literature, there are two other approaches, closer to ours in spirit, to obtain higher-order regularity and explicit formulae for the derivatives of the invariant measure. It is certainly interesting to discuss the differences and similarities between those papers and ours.
\\The framework of \cite{GL}, a far-reaching generalization of \cite{KL}, allows to obtain high smoothness, as well as explicit formulae at any order, for the whole discrete spectrum of a family of operators, if a uniform spectral gap is present and suitable Taylor expansions are satisfied. This approach is usually referred to as weak spectral perturbation theory, in contrast to classical perturbation theory (\`a la Kato, \cite{Kato}). It is applicable for a wide class of deterministic and random perturbations of uniformly hyperbolic systems. However, we notice that the type of random perturbations considered in this paper are limited to zero-noise limit, and that for the very general type of random walk studied there (see \cite[p.6]{GL}), only Lipschitz continuity results are formulated.
\\In \cite{JS}, a
construction to get high differentiability and explicit formulae for higher
derivatives of perturbations of fixed points of operators in Banach
spaces is presented and an application to linear response in uniformly expanding systems is
shown. Among its strengths, this approach allows to obtain results both in deterministic and random situations, especially in the quenched case (see \cite[Chapter 5]{Sedro}); it may be applied to non-linear operators; although this is not obvious at first sight, it may also be used to study the regularity w.r.t parameters of the whole discrete spectrum \cite{Sedro2}. Let us remark that it may also apply in systems where no spectral gap is present. In our opinion, the main weakness of this approach is in the very heavy notation one has to digest in order to implement it.
\\In this paper, we focus on the first and second term of the Taylor's
development of the response using statements which are somewhat simpler than
the ones presented in \cite{JS} or \cite{GL}, but lighter notation-wise. Our approach is flexible enough to be applied both to
random and deterministic perturbations of a given system. Furthermore, it only requires the unperturbed system to exhibit convergence to equilibrium and a well-defined resolvent, (see Assumption LR2 and LR3), in contrast to the exponential mixing requirement for the whole family of operators in \cite{GL}. The existence of a quadratic response and related formulae for zero noise limits of
deterministic systems was known (it can be obtained via \cite{GL}, see also \cite{Li2}), but those results in the case where the noise is an intrinsic part of the model, as far as we know, are new.

\noindent \textbf{Plan of the paper and main results. }In Sections \ \ref%
{sec:linresp} and \ref{sec:quadresp} we prove two abstract theorems, giving a
framework of general assumptions on the system and its associated transfer
operator in which the expansion \eqref{Quad} can be obtained. We also
show explicit formulas for $R$ and $Q$ ($R$ and $Q$ will belong to suitable
normed vector spaces of measures or distributions). One of the required assumptions
is the existence of certain resolvent for the unperturbed transfer operators associated to our
systems. In Section \ref{resolvsec} we show how this existence can be
deduced by suitable regularization properties of the transfer operators we
consider (Lasota Yorke inequalities on suitable measure spaces or the
regularization brought by the effect of noise e.g.). \ Our framework is flexible enough to apply both to random and
deterministic systems and in Sections \ref{secnoise} and \ref{secmap} \ we
give examples of such applications.

\section{First derivative, linear response\label{sec:linresp}}

In this section we show a general result for the linear response of fixed
points of Markov operators under suitable perturbations. The result is made
to be applied to transfer operators of dynamical systems and suitable
perturbations. Let $X$ be a compact metric space. Let us consider the space
of signed Borel measures on $X$, $BS(X)$. \ In the following we consider
three normed vectors spaces of signed Borel measures on $X.$ The spaces $%
(B_{ss},\|~\|_{ss})\subseteq (B_{s},\|~\|_{s})\subseteq
(B_{w},\|~\|_{w})\subseteq BS(X)$ with norms satisfying%
\begin{equation*}
\|~\|_{w}\leq \|~\|_{s}\leq \|~\|_{ss}.
\end{equation*}%
We remark that, a priori, these spaces can be taken equal. Their precise choice
depends on the type of system and perturbation under study. 

We will assume that the linear form $\mu\to\mu(X)$ is continuous on $B_i$, for $i\in\{ss,s,w\}$.
Since we will consider   Markov operators \footnote{%
	A Markov operator is a linear operator preserving positive measures and such
	that for each positive measure $\mu$, it holds $[L(\mu )](X)=\mu (X)$.} acting on these spaces, the
following (closed) spaces $V_{ss}\subseteq V_{s}\subseteq V_{w}$ of \ zero average
measures defined as:%
\begin{equation*}
V_{i}:=\{\mu \in B_{i}|\mu (X)=0\}
\end{equation*}%
where $i\in \{ss,s,w\}$, will play an important role. If $A,B$ are two
normed vector spaces and $T:A\rightarrow B$ we denote the mixed norm $\Vert
T\Vert _{A\rightarrow B}$ as 
\begin{equation*}
\Vert T\Vert _{A\rightarrow B}:=\sup_{f\in A,\Vert f\Vert _{A}\leq 1}\Vert
Tf\Vert _{B}.
\end{equation*}

Suppose hence we have a one parameter family of Markov
operators $L_{\delta }.$ The following theorem is similar to the linear
response theorem for regularizing transfer operators used in \cite{GG}, the
present statement is adapted to a general application on both deterministic
and random systems.

\begin{theorem}[Linear Response]
\label{thm:linresp} \label{th:linearresponse} Suppose that the family of bounded Markov
operators $L_{\delta }:B_{i}\rightarrow B_{i},$ where $i\in \{ss,s,w\}$
satisfy the following:

\begin{itemize}
\item[(LR1)] (regularity bounds) for each $\delta \in \left[ 0,\overline{%
\delta }\right) $ there is $h_{\delta }\in B_{ss}$, a probability measure
such that $L_{\delta }h_{\delta }=h_{\delta }$. Furthermore, there is $%
M\geq 0$ such that for each $\delta \in %
\left[ 0,\overline{\delta }\right) $
$$\Vert h_{\delta }\Vert _{ss}\leq M.$$ 

\item[(LR2)] (convergence to equilibrium for the unperturbed operator) There is a sequence $a_n\to 0$ such that for each $g\in V_{ss}$%
\begin{equation*}
\Vert L_{0}^{n}g\Vert _{s}\leq a_n||g||_{ss};
\end{equation*}

\item[(LR3)] (resolvent of the unperturbed operator) $(Id-L_{0})^{-1}:=%
\sum_{i=0}^{\infty }L_{0}^{i}$ is a bounded operator $V_{w}\rightarrow V_{w} 
$.

\item[(LR4)] (small perturbation and derivative operator) There is $K\geq 0$
such that $\left\vert |L_{0}-L_{\delta }|\right\vert _{B_{s}\rightarrow
B_{w}}\leq K\delta ,$ and $\left\vert |L_{0}-L_{\delta }|\right\vert
_{B_{ss}\rightarrow B_{s}}\leq K\delta $. There is ${\dot{L}h_{0}\in V_{w}}$
such that%
\begin{equation}
\underset{\delta \rightarrow 0}{\lim }\left\Vert \frac{(L_{\delta}-L_{0})}{%
\delta }h_{0}-\dot{L}h_{0}\right\Vert _{w}=0.  \label{derivativeoperator}
\end{equation}
\end{itemize}

Then we have the following Linear Response formula%
\begin{equation}
\lim_{\delta \rightarrow 0}\left\Vert \frac{h_{\delta }-h_{0}}{\delta }%
-(Id-L_{0})^{-1}\dot{L}h_{0}\right\Vert _{w}=0.  \label{linresp}
\end{equation}
\end{theorem}

\begin{remark}
The choice for the three spaces depends on the system and the perturbation
considered. We remark that the space where the response is defined is the
same as the one where the derivative operator is defined. Concrete examples
will be shown in Sections \ref{secnoise} and \ref{secmap}.
\end{remark}

\begin{remark}
\label{rmkLM}The convergence to equilibrium assumption at Item $(LR2)$ is required only for the 
\emph{unperturbed operator} $L_{0}$. It is sometimes not trivial to 
prove, but is somehow expected in systems having some sort of
indecomposability and chaotic behavior (topological mixing, expansion,
hyperbolicity or noise e.g.). In \cite{GG} there are several examples of
verification of this condition by different methods in systems with additive
noise.
\end{remark}


\begin{remark}
The regularity bounds asked in Assumption $(LR1)$ are easily verified in
systems satisfying some regularization properties, like a Lasota Yorke
inequality or the effect of noise (see Section \ref{secnoise}).
\end{remark}

\begin{remark}
The assumption $(LR3)$ on the existence of the resolvent may be harder to
verify. This is asked only for the unperturbed transfer operator, allowing a
large class of perturbations. This technical point is quite useful and
different from the kind of assumptions required in other previous approaches
(e.g.\ \cite{GL}). In many systems, it will result from the presence of a
spectral gap (compactness or quasi-compactness of the transfer operator
acting on $B_{w}$). In Section \ref{resolvsec} we will prove this assumption
in the case of regularizing operators, which include systems with additive
noise.
\end{remark}

\begin{remark}
As remarked in the introduction, a family of operators {might} fail to have
linear response, sometimes because of lack of hyperbolicity, sometimes
because of the non smoothness of the kind of perturbation which is
considered along the family (see e.g. \cite{BB}). In particular this is
related to the type of convergence of the derivative operator 
\begin{equation}
\dot{L}f=\lim_{\delta \rightarrow 0}\frac{(L_{\delta }-L_{0})}{\delta }f.
\end{equation}%
In deterministic systems and related transfer operators, if the system is
perturbed by moving its critical values or discontinuities, this will result
in a bad perturbation of the associated transfer operators, and the limit
defining \ $\dot{L}$ will not converge, unless we consider very coarse
topologies in which the resolvent operator might not be a bounded operator.
\end{remark}

We are ready to prove the main general statement.

\begin{proof}[Proof of Theorem \protect\ref{th:linearresponse}]
Let us first prove that under the assumptions the system has strong
statistical stability in $B_{s}$, that is 
\begin{equation}
\lim_{\delta \rightarrow 0}\Vert h_{\delta }-h_{0}\Vert _{s}=0.
\end{equation}

Let us consider for any given $\delta $ a probability measure $h_{\delta }$
such that $L_{\delta }h_{\delta }=h_{\delta }$. Thus%
\begin{eqnarray*}
\Vert h_{\delta }-h_{0}\Vert _{s} &\leq &\Vert L_{\delta }^{N}h_{\delta
}-L_{0}^{N}h_{0}\Vert _{s} \\
&\leq &\Vert L_{\delta }^{N}h_{\delta }-L_{0}^{N}h_{\delta }\Vert _{s}+\Vert
L_{0}^{N}h_{\delta }-L_{0}^{N}h_{0}\Vert _{s}.
\end{eqnarray*}%
Since $h_{\delta },h_{0}$ are probability measures, $h_{\delta }-h_{0}\in
V_{ss}$ and by $(LR1)$, $\Vert h_{\delta }-h_{0}\Vert _{ss}\leq 2M$ then we have 
\begin{equation*}
\Vert h_{\delta }-h_{0}\Vert _{s}\leq \Vert L_{\delta }^{N}h_{\delta
}-L_{0}^{N}h_{\delta }\Vert _{s} + Q(N)
\end{equation*}%
with $Q(N)=2a_n M\rightarrow 0$, not depending on $\delta$, because of the assumption $(LR2)$.


 Next we
rewrite the operator sum $L_{0}^{n}-L_{\delta }^{n}$ telescopically 
\begin{equation*}
(L_{0}^{N}-L_{\delta }^{N})=\sum_{k=1}^{N}L_{0}^{N-k}(L_{0}-L_{\delta
})L_{\delta }^{k-1}
\end{equation*}%
so that 
\begin{eqnarray*}
(L_{\delta }^{N}-L_{0}^{N})h_{\delta }
&=&\sum_{k=1}^{N}L_{0}^{N-k}(L_{\delta }-L_{0})L_{\delta }^{k-1}h_{\delta }
\\
&=&\sum_{k=1}^{N}L_{0}^{N-k}(L_{\delta }-L_{0})h_{\delta }.
\end{eqnarray*}%
The assumption that $\Vert h_{\delta }\Vert _{ss}\leq M,$ together with the
small perturbation assumption (LR4) implies that $\Vert (L_{\delta
}-L_{0})h_{\delta }\Vert _{s}\leq \delta KM$ as $\delta \rightarrow 0.$ {Thus%
} 
\begin{equation}
\Vert h_{\delta }-h_{0}\Vert _{s}\leq Q(N)+ NM_2[\delta KM]
\label{eq:strongstability}
\end{equation}
where $M_2=\max(1, ||L_0||^N_{B_s\to B_s}).$
Choosing first $N$ big enough to let $Q(N)$ be close to $0$ and then $\delta $ small enough we can make $%
\Vert h_{\delta }-h_{0}\Vert _{s}$ as small as wanted, proving the stability
in $B_{s}$.

Let us now consider $(Id-L_{0})^{-1}$ as a continuous operator $%
V_{w}\rightarrow V_{w}$. Remark that since $\dot{L}h_{0}\in V_{w},$ the
resolvent can be computed at $\dot{L}h_{0}$ . Now we are ready to prove the
main statement. By using that $h_{0}$ and $h_{\delta }$ are fixed points of
their respective operators we obtain that%
\begin{equation*}
(Id-L_{0})\frac{h_{\delta }-h_{0}}{\delta }=\frac{1}{\delta }(L_{\delta
}-L_{0})h_{\delta }.
\end{equation*}%
By applying the resolvent to both sides 
\begin{eqnarray*}
(Id-L_{0})^{-1}(Id-L_{0})\frac{h_{\delta }-h_{0}}{\delta } &=&(Id-L_{0})^{-1}%
\frac{L_{\delta }-L_{0}}{\delta }h_{\delta } \\
&=&(Id-L_{0})^{-1}\frac{L_{\delta }-L_{0}}{\delta }h_{0}+(Id-L_{0})^{-1}%
\frac{L_{\delta }-L_{0}}{\delta }(h_{\delta }-h_{0})
\end{eqnarray*}%
we obtain that the left hand side is equal to $\frac{1}{\delta }(h_{\delta
}-h_{0})$. Moreover, with respect to right hand side we observe that,
applying assumption $(LR4)$ eventually, as $\delta \rightarrow 0$%
\begin{equation*}
\left\Vert (Id-L_{0})^{-1}\frac{L_{\delta }-L_{0}}{\delta }(h_{\delta
}-h_{0})\right\Vert _{w}\leq \Vert (Id-L_{0})^{-1}\Vert _{V_{w}\rightarrow
V_{w}}K\Vert h_{\delta }-h_{0}\Vert _{s}
\end{equation*}%
which goes to zero thanks to \eqref{eq:strongstability}. Thus considering
the limit $\delta \rightarrow 0$ we are left with 
\begin{equation*}
\lim_{\delta \rightarrow 0}\frac{h_{\delta }-h_{0}}{\delta }=(Id-L_{0})^{-1}%
\dot{L}h_{0}.
\end{equation*}%
converging in the $\Vert \cdot \Vert _{w}$ norm, which proves our claim.
\end{proof}

\section{The second derivative\label{sec:quadresp}}

In this section we show how the previous approach can give us information on
the second derivative and the second order term of the response to a
perturbation.

Consider a further space $(B_{ww},\|~\|_{ww})$ such that $%
(B_{w},\|~\|_{w})\subseteq (B_{ww},\|~\|_{ww})\subseteq BS(X)$ and 
\begin{equation*}
\|~\|_{ww}\leq \|~\|_{w}.
\end{equation*}%
on which the linear form $\mu\to\mu(X)$ is continuous.
Let us also consider the space of zero average measures in $B_{ww}$%
\begin{equation*}
V_{ww}:=\{\mu \in B_{ww}|\mu (X)=0\}.
\end{equation*}%
We now prove an abstract response result for the second derivative.

\begin{theorem}[Quadratic term in the response]
\label{thm:quadresp} Let $(L_{\delta })_{\delta \in \lbrack 0,\overline{%
\delta }]}:B_{i}\rightarrow B_{i}$, $i\in \{ss,...,ww\}$ be a family of
Markov operators as in the previous theorem. \ Assume furthermore that:

\begin{enumerate}
\item[(QR1)] The derivative operator $\dot{L}$ admits a bounded extension $%
\dot{L}:B_{w}\rightarrow V_{ww}$, such that 
\begin{equation}
\left\Vert \frac{1}{\delta }(L_{\delta}-L_{0})-\dot{L}\right\Vert
_{w\rightarrow ww}\underset{\delta \rightarrow 0}{\longrightarrow }0.
\label{eq:unifcvderivop}
\end{equation}

\item[(QR2)] There exists a "second derivative operator" at $h_{0}$, i.e. $\ 
\ddot{L}h_{0}\in V_{ww}$ such that 
\begin{equation}
\left\Vert \dfrac{(L_{\delta }-L_{0})h_{0}-\delta \dot{L}h_{0}}{\delta ^{2}}-%
\ddot{L}h_{0}\right\Vert _{ww}\underset{{\delta }\rightarrow 0}{%
\longrightarrow }0.  \label{def:2ndderivop}
\end{equation}

\item[(QR3)] The resolvent operator $(Id-L_{0})^{-1}$ admits a bounded
extension as an operator $V_{ww}\rightarrow V_{ww}$.
\end{enumerate}

Then one has the following: the map $\delta \in \lbrack 0,\overline{\delta }%
]\mapsto h_{\delta }\in B_{ss}$ has an order two Taylor's expansion at $%
\delta =0$, with 
\begin{equation}
\left\Vert \frac{h_{\delta }-h_{0}-\delta (Id-L_{0})^{-1}\dot{L}h_{0}}{%
\delta ^{2}}-(Id-L_{0})^{-1}\left[ \ddot{L}h_{0}+\dot{L}(Id-L_{0})^{-1}\dot{L%
}h_{0}\right] \right\Vert _{ww}\underset{{\delta }\rightarrow 0}{%
\longrightarrow }0.  \label{eq:quadresp}
\end{equation}
\end{theorem}

\begin{remark}
We require the first derivative operator (see $($\ref{eq:unifcvderivop}$)$)
to be defined not only at the stationary measure, but on the whole space $%
B_{w}$ with convergence in the $ww$ topology, while for the second
derivative operator $\ $(see $($\ref{def:2ndderivop}$)$) we need it to be
defined only at $h_{0}$. We also remark that the Quadratic response
converges in the same norm in which the second derivative operator converges.
\end{remark}

\begin{proof}
We write, for $\delta \not=0$, 
\begin{align}
(Id-L_{0})\frac{h_{\delta }-h_{0}-\delta (Id-L_{0})^{-1}\dot{L}h_{0}}{\delta
^{2}}& =\frac{1}{\delta ^{2}}\left[ (Id-L_{0})(h_{\delta }-h_{0})-\delta 
\dot{L}h_{0}\right]  \notag \\
& =\frac{1}{\delta ^{2}}\left[ (L_{\delta }-L_{0})h_{\delta }-\delta \dot{L}%
h_{0}\right]  \notag \\
& =\frac{1}{\delta ^{2}}\left( L_{\delta }-L_{0}\right) (h_{\delta }-h_{0})+%
\frac{1}{\delta ^{2}}\left[ (L_{\delta }-L_{0})h_{0}-\delta \dot{L}h_{0}%
\right] .  \label{eqx}
\end{align}%
By assumption $(QR2)$, in $($\ref{eqx}$)$ the second term of the right-hand
term, 
\begin{equation*}
\frac{1}{\delta ^{2}}\left[ (L_{\delta }-L_{0})h_{0}-\delta \dot{L}h_{0}%
\right] \underset{\delta \rightarrow 0}{\longrightarrow }\ddot{L}h_{0}
\end{equation*}%
in the $V_{ww}$-norm. \newline
The first in $($\ref{eqx}$)$ can be rewritten as%
\begin{eqnarray}
\frac{\left( L_{\delta }-L_{0}\right) }{\delta }\frac{(h_{\delta }-h_{0})}{%
\delta } &=&\left( \frac{\left( L_{\delta }-L_{0}\right) }{\delta }-\dot{L}%
\right) \frac{(h_{\delta }-h_{0})}{\delta }+\dot{L}\left( \frac{(h_{\delta
}-h_{0})}{\delta }\right)  \notag \\
&=&\left( \frac{\left( L_{\delta }-L_{0}\right) }{\delta }-\dot{L}\right) [%
\dot{h}-\dot{h}+\frac{(h_{\delta }-h_{0})}{\delta }]+\dot{L}\left( \frac{%
(h_{\delta }-h_{0})}{\delta }\right) ,  \label{eqst}
\end{eqnarray}%
where $\dot{h}:=\lim_{\delta \rightarrow 0}\dfrac{h_{\delta }-h_{0}}{\delta }%
\in V_{w}$ (it is well-defined by Theorem \ref{th:linearresponse}). By
uniform convergence of $\frac{\left( L_{\delta }-L_{0}\right) }{\delta }$
towards the derivative operator $\dot{L}$ in \eqref{eq:unifcvderivop}, the
first summand in the right hand term of \ref{eqst} converges in $V_{ww}$ as $%
\delta \rightarrow 0$ to $0$. \newline
For the second summand in \ref{eqst} we write 
\begin{equation}
\left\Vert \dot{L}\left[ \frac{(h_{\delta }-h_{0})}{\delta }\right] -\dot{L}%
(Id-L_{0})^{-1}\dot{L}h_{0}\right\Vert _{ww}\leq \Vert \dot{L}\Vert
_{w\rightarrow ww}\left\Vert \frac{h_{\delta }-h_{0}}{\delta }%
-(Id-L_{0})^{-1}\dot{L}h_{0}\right\Vert _{w}
\end{equation}%
which goes to $0$ as $\delta \rightarrow 0$ thanks to Theorem \ref%
{th:linearresponse}. \ Thus, we have that in the $B_{ww}$ norm 
\begin{equation}
(Id-L_{0})\frac{h_{\delta }-h_{0}-\delta (Id-L_{0})^{-1}\dot{L}h_{0}}{\delta
^{2}}\underset{\delta \rightarrow 0}{\longrightarrow }\ddot{L}h_{0}+\dot{L}%
(Id-L_{0})^{-1}\dot{L}h_{0}.
\end{equation}

We conclude by applying the resolvent $(Id-L_{0})^{-1}$, well defined on $%
V_{ww}$.
\end{proof}

\section{Existence of the resolvent for $L_{0}$ and regularization.\label%
{resolvsec}}

In this section we show how the presence of some regularization and
compactness allows to show that the resolvent operator $(Id-L_{0})^{-1}$ is
well defined and continuous on the space of zero average measures. The
following statement, is a version of a classical tool to obtain spectral gap
in systems satisfying a Lasota Yorke inequality.
It allows one to estimate the contraction rate of zero
average measures, and imply spectral gap when applied to Markov operators.
Let us consider a transfer operator $L_0$ acting on two normed vector spaces
of complex or signed measures $(B_{s},\|~\|_{s}),~(B_{w},\|~\|_{w}),$ $%
B_{s}\subseteq B_{w}$ with $\|~\|_{s}\geq \|~\|_{w}$. We furthermore assume
that $\mu\mapsto\mu(X)$ is continuous in the $\|.\|_s$ and $\|.\|_w$
topologies, and let $V_i:=\{\mu \in B_{i},\mu (X)=0\}$, $i\in\{w,s\}$.

\begin{theorem}
\label{gap}Suppose:

\begin{enumerate}
\item (Lasota Yorke inequality). For each $g\in B_{s}$%
\begin{equation*}
\|L_0^{n}g\|_{s}\leq A\lambda _{1}^{n}\|g\|_{s}+B\|g\|_{w};
\end{equation*}

\item (Mixing) for each $g\in V_{s}$, it holds 
\begin{equation*}
\lim_{n\rightarrow \infty }\|L_0^{n}g\|_{w}=0;
\end{equation*}

\item (Compact inclusion) The image of the closed unit ball in $B_{s}$ under 
$L_0$ is relatively compact in $B_{w}$.

\end{enumerate}

Under these assumptions, we have 

\begin{enumerate}
	\item $L_0$ admits a unique fixed point in $h\in B_s$,
satisfying $h(X)=1$.
 \item There are $C>0,\rho<1$ such that for all $f\in V_{s}$ and $m$
large enough, 
\begin{equation}
\|L_0^{m}f\|_{s}\leq C\rho^{m}\|f\|_{s}.  \label{gap2}
\end{equation}
\end{enumerate}
\end{theorem}

\begin{proof}
First we prove the existence of the fixed point. By Hennion theorem \cite{Hennion}, the
essential spectral radius of $L_0$ on $B_s$ is $\le \lambda_1$. Hence for
any $r\in(\lambda_1,1)$, there are only isolated eigenvalues of finite
multiplicity in the annulus $r\leq|z|\leq 1$. Let $\theta\not\in 2\pi\mathbb{%
Z}$, such that $e^{i\theta}$ is an eigenvalue of $L_0$. Let $f\in B_s$ be an
associated eigenfunction. Since $L_0$ is Markov, we have 
\begin{equation*}
\int_Xfdm=\int_X L_0fdm=\int_X e^{i\theta}fdm,
\end{equation*}
hence $\int_X fdm=0$. But then by the mixing assumption, we have $%
\|f\|_w=\|L_0^nf\|_w\underset{n\rightarrow \infty}{\longrightarrow}0$, which
is a contradiction. \newline
We conclude that $1$ is the only eigenvalue on the unit circle. Furthermore,
it is simple, since if there are $h_1$, $h_2$ satisfying $L_0h_1=h_1$, $%
L_0h_2=h_2$ and normalized so that $\int h_1dm=\int h_2dm=1$, then $%
h_1-h_2\in V_s$ and hence by the mixing assumption $\|h_1-h_2\|_w=0$. 
\newline
Now we turn to the estimate on the rate of mixing. What follows is
essentially an adaptation of Hennion's proof\footnote{%
We are grateful to the anonymous referee for communicating this argument to
us.} on the bound on the essential spectral radius. Let $S_s$ be the closed
unit ball in $V_s$. By the compactness assumption, for each $\epsilon>0$,
there is an $\epsilon$-net $G_\epsilon$ for $\overline {L_0(S_s)}$ in $%
\|.\|_w$, which implies that for any $f\in S_s$ we may find $g\in \overline{%
L_0(S_s)}\cap G_\epsilon$ so that $\|L_0f-g\|_w\leq\epsilon$. By the
Lasota-Yorke inequality we then have, for any $m\in\mathbb{N}$, 
\begin{align*}
\|L_0^mf\|_s=\|L_0^{m-1}L_0f\|_s&\leq\|L_0^{m-1}g\|_s+\|L_0^{m-1}(L_0f-g)\|_s
\\
&\leq \|L_0^{m-1}g\|_s+2A\lambda_1^{m-1}\|L_0\|+B\epsilon.
\end{align*}
We then write $m-1=n+k$, so that 
\begin{equation*}
\|L_0^{m-1}g\|_s=\|L_0^nL_0^kg\|_s\leq
A\lambda_1^n\|L_0^kg\|_s+B\|L_0^kg\|_w\leq
A^2\lambda_1^{n+k}+AB\lambda_1^n\|g\|_w+B\|L_0^kg\|_w,
\end{equation*}
from which we get 
\begin{equation*}
\|L_0^mf\|_s\le
A^2\lambda_1^{n+k}+AB\lambda_1^n\|g\|_w+B\|L_0^kg\|_w+2A\lambda_1^{m-1}\|L_0%
\|+B\epsilon.
\end{equation*}
We fix $\epsilon$ small enough to get $B\epsilon<1/4$. By the mixing
assumption, we may take $k$ large enough to get $\sup_{g\in
G_\epsilon}\|L_0^kg\|_w\leq 1/4B$. We may then choose $n$ large enough to
obtain 
\begin{equation*}
A^2\lambda_1^{n+k}+AB\lambda_1^n\sup_{g\in
G_\epsilon}\|g\|_w+2A\lambda_1^{m-1}\|L_0\|\leq 1/4,
\end{equation*}
from which follows $\|L_0^mf\|_s\leq 3/4$. Hence the result.
\end{proof}

By this result, the existence of the resolvent follows easily.

\begin{corollary}
\label{resolvent} Under the assumptions of Theorem \ref{gap}, the resolvent $%
(Id-L_{0})^{-1}:$ $V_{s}\rightarrow V_{s}$ is defined and continuous.
\end{corollary}

\begin{proof}
Let $f\in V_{s}$ then, by definition, $(Id-L_{0})^{-1}f=\sum_{0}^{\infty
}L_{0}^{i}f$. By the Markov assumption $L_{0}^{i}f\in V_{s}$ for $i\geq 1$.
Since \eqref{gap2} holds and $\sum_{1}^{\infty }C\rho^{n}<\infty ,$ the sum $%
\sum_{1}^{\infty }L_{0}^{i}f$ \ converges in $V_{s}$ with respect to the $%
\|.\|_{s}$ norm, and $\|(Id-L_{0})^{-1}\|_{V_{s}\rightarrow V_{s}}\leq
\sum_{1}^{\infty }C\rho^{n}$.
\end{proof}

\begin{remark} \label{rem:compactness} 
In the case where $B_s=B_w=B$, Theorem \ref{gap} still apply. In this case, we get the result that an operator that is power-bounded on $B$, compact and mixing has a unique (up to normalization) fixed point, and has exponential mixing.
\end{remark}

An important case in which the assumptions of Theorem \ref{gap} are
satisfied is systems with additive noise, for which the transfer operator
often satisfies a regularization assumption:
Let $B_s\subset B_w$ be two Banach spaces with $%
\|.\|_w\leq\|.\|_s$. We say $L_{0}$ is regularizing from $B_{w}$ to $B_{s}$ if $%
L_{0}:B_{w}\rightarrow B_{s}$ is continuous, i.e if there is $B>0$ such that the inequality%
\begin{equation*}
\|L_{0}f\|_{s}\leq B\|f\|_{w}
\end{equation*}%
is satisfied. If moreover a weak boundedness assumption is verified on $B_w$%
, that is if there is some $C>0$ such that 
\begin{equation*}
\sup_{n}\|L_{0}^{n}f\|_{w}:=C\|f\|_w
\end{equation*}%
for all $n$ and $f\in B_w$, then we have the Lasota Yorke inequality%
\begin{equation*}
\|L_0^{n}g\|_{s}\leq CB\|g\|_{w}
\end{equation*}%
holding for each $n$. If the compactness assumption $(2)$ in Theorem \ref%
{gap} is satisfied, this easily implies (by Hennion theorem \cite{Hennion}) that on the
strong space, the operator $L_0$ only has discrete spectrum. 

\begin{corollary}
\label{resolventnoise}If $L_{0}$ is regularizing from $B_{w}$ to $B_{s}$,
and if Assumptions (2) and (3) in Theorem \ref{gap} holds, then the
resolvent $(Id-L_{0})^{-1}$ is defined and continuous also on $V_{w}$.
Furthermore, let $B_{ww}\supseteq B_{w}$ as at beginning of Section \ref%
{sec:quadresp}. Suppose $L_{0}$ is regularizing from $B_{ww}$ to $B_{w}$
i.e. $L_{0}:B_{ww}\rightarrow B_{w}$ is continuous, then $(Id-L_{0})^{-1}$
is defined and continuous on $V_{ww}$ too.
\end{corollary}

\begin{proof}
Let $f\in V_{w}$. Since $(Id-L_{0})^{-1}f=f+\sum_{1}^{\infty }L_{0}^{i}f$,
we get 
\begin{equation*}
\|(Id-L_{0})^{-1}f\|_{w}\leq \|f\|_{w}+\|L_{0}\|_{w\rightarrow
s}\sum_{i=0}^{\infty }\|L_{0}^{i}\|_{s\rightarrow s}\|f\|_w<\infty
\end{equation*}
by Corollary \ref{resolvent}. This shows that $(Id-L_{0})^{-1}=Id{+}%
\sum_{1}^{\infty }L_{0}^{i}$ is a continuous operator $V_{w}\rightarrow
V_{w} $. In case $L_{0}:V_{ww}\rightarrow V_{w}$ is continuous we can repeat
the same proof with $V_{ww}$ and $V_{w}$ in the place of $V_{w}$ and $V_{s},$%
\ obtaining that is a continuous operator $V_{ww}\rightarrow V_{ww}$.
\end{proof}

\section{Linear and Quadratic response in systems with additive noise\label%
{secnoise}}

In this section, we consider a non-singular map $T$, defined on the circle $%
\S ^{1}$, perturbed by composition with a $C^{3}$ diffeomorphism near
identity $D_{\delta }$ (in a sense explained precisely in \eqref{eq:detpertisC1} and 
\eqref{eq:detpertisC2}), and an additive noise with smooth
kernel $\rho_\xi$. For example, one may take a Gaussian kernel 
\begin{equation}
\rho _{\xi }=\dfrac{e^{-x^{2}/2\xi ^{2}}}{\sqrt{2\pi }\xi }.
\end{equation}%
In other words we consider a random dynamical system, corresponding to
the stochastic process $(X_{n})_{n\in \mathbb{N}}$ defined by 
\begin{equation}
X_{n+1}=D_{\delta }\circ T(X_{n})+\Omega _{n}\mod 1  \label{eq:syswaddnoise}
\end{equation}%
where $(\Omega _{n})_{n\in \mathbb{N}}$ are i.i.d random variables generated
by the kernel $\rho$. \medskip

To this system we associate the annealed transfer operator defined by 
\begin{equation}
L _{\delta }:=\rho _{\xi }\ast L_{D_{\delta }\circ T}
\label{def:annealedtransferop}
\end{equation}%
(see section \ref{sec:regineq} for the proper definition of the convolution $%
\ast $ in this context) where $L_{D_{\delta }\circ T}=L_{D_{\delta }}\circ
L_{T}$ is the transfer operator (the pushforward map) associated to the
deterministic map $D_{\delta }\circ T$ (see \cite{Viana}, Section 5 for more
details about transfer operators associated to this kind of systems).

Our goal in this section is to show how this family of systems exhibits
linear and quadratic response, as $\delta\to 0$ by applying Theorems \ref%
{thm:linresp} and \ref{thm:quadresp}: in the following, we thus verify that
the family of transfer operators $(L_{\delta })_{\delta \in \lbrack
-\epsilon ,\epsilon ]}$ satisfy the assumptions of these two Theorems.

\subsection{Convolution with Gaussian kernel and regularization inequalities
on the circle}

\label{sec:regineq}

In this section we show the regularization properties of the convolution
product of a Gaussian kernel and a finite order distribution on the circle.

Let $\rho _{\xi }=\dfrac{1}{\sqrt{2\pi }\xi }e^{-x^{2}/2\xi ^{2}}$ be the
Gaussian kernel. It is not \emph{a priori} obvious how one can define the
convolution product of the Gaussian kernel and a probability density on the
circle, as the former is not one-periodic. \newline
To that effect we start by recalling the definition of the \emph{%
periodization} of a Schwartz function\footnote{%
Recall that $\rho :\mathbb{R}\rightarrow \mathbb{R}$ is a Schwartz function
if it is a $C^{\infty }$ function that satisfies, for any $(n,m)\in \mathbb{N%
}^{2}$, 
\begin{equation*}
|x|^{n}\rho ^{(m)}(x)\underset{|x|\rightarrow \infty }{\longrightarrow }0.
\end{equation*}%
The set of Schwartz function is traditionally denoted by $\mathcal{S}(%
\mathbb{R})$.} $\rho $:

\begin{definition}
If $\rho:\mathbb{R}\to\mathbb{R}$ is a Schwartz function, we define its
periodization $\tilde{\rho}:\S ^1\to\mathbb{R}$ by 
\begin{equation*}
\tilde{\rho}(x)=\sum_{k\in\mathbb{Z}}\rho(x+k)
\end{equation*}
\end{definition}

It is clear that the latter series converges uniformly on any bounded
interval. As the same holds for its derivatives, it follows that $\tilde{\rho%
}$ defines a $C^\infty$ function on the circle. Thus, for $f\in L^1(\S ^1)$,
we can define $\rho \ast f$ as follows:

\begin{definition}
Let $f\in L^{1}(\S ^{1})$ and $\rho \in \mathcal{S}(\mathbb{R})$. We define
the convolution $\rho \ast f$ by the formula 
\begin{equation}
\rho \ast f(x):=\int_{\S ^1}\tilde\rho (x-y)f(y)dy  \label{def:convprod}
\end{equation}
\end{definition}

The next proposition is an obvious consequence of the definition. We omit
the proof.

\begin{proposition}
Let $\rho$ and $f$ be as before. The convolution $\rho\ast f$ has the
following properties:

\begin{enumerate}
\item $\rho \ast f:\S ^{1}\rightarrow \mathbb{R}$ is $C^{\infty }$, with $%
(\rho \ast f)^{(k)}=\rho ^{(k)}\ast f$ for any $k\in\mathbb{N}$.

\item One has the following regularization inequality: for each $k\in 
\mathbb{N}$, 
\begin{equation}
\Vert \rho \ast f\Vert _{C^{k}}\leq \Vert \rho \Vert _{C^{k}}\Vert f\Vert
_{L^{1}}.  \label{eq:regineqI}
\end{equation}
\end{enumerate}
\end{proposition}

\begin{remark}
We notice that the same construction could be carried out for $\rho\in
C^\infty(\S ^1)$, or for $\rho\in C^\infty(\mathbb{R})$ decaying
sufficiently fast (e.g so that $|\rho^{(j)}(x+k)|\leq 1/k^2$ for all $j\in%
\mathbb{N}$ and $x$.)
\end{remark}

Now, we extend the previous definition of convolution to the case of a
distribution on the circle, by duality. We introduce the space of
distributions on the circle $\mathcal{D}^{\prime }(\S ^1)$, to write the

\begin{definition}
\label{def:convprodII} Let $f\in \mathcal{D}^{\prime }(\S ^{1})$, and $\rho $
be a Schwartz function. Then $\rho \ast f$ is the distribution on the circle
defined for all $\phi \in C^{\infty }(\S ^{1})$ by 
\begin{equation}
\langle \rho \ast f,\phi \rangle =\langle f,\bar\rho \ast \phi \rangle,
\label{eq:convprodII}
\end{equation}%
where $\bar\rho$ is defined by $\bar\rho(x):=\rho(-x)$ for $x\in\S^1$.
\\We also define the space of distributions of order N, $\mathcal{D}_N(\S %
^1):=\{f\in\mathcal{D}^{\prime }(\S ^1), \|f\|_{\mathcal{D}_N}<\infty \}$,
where 
\begin{equation*}
\|f\|_{\mathcal{D}_N}:=\sup_{\QATOP{{\|\phi\|_{C^N}\leq 1}}{{\phi\in C^N(\S %
^1)}}}|\langle f,\phi\rangle|
\end{equation*}
and 
\begin{equation*}
W^{-N,1}(\S ^{1}):=\{f\in \mathcal{D}^{\prime }(\S ^{1}),~\exists F\in L^{1}(%
\S ^{1}),~F^{(N)}=f~\text{in the sense of distributions}\}.
\end{equation*}
\end{definition}

Now assume that $f\in W^{-N,1}(\S ^{1})$. Then one has the following result:

\begin{proposition}
\label{prop:regineqII} Let $f\in W^{-N,1}(\S ^{1})$ and $\rho $ be a
Schwartz function. Then $f$ induces a distribution of order $N$, and one has
that: $\rho \ast f$ is a $C^{\infty }$, one-periodic function, and for any $%
k\in \mathbb{N}$, 
\begin{equation}
\Vert \rho \ast f\Vert _{C^{k}}\leq C\Vert \rho \Vert _{C^{N+k}}\Vert f\Vert
_{\mathcal{D}_{N}}.  \label{eq:regineqII}
\end{equation}
\end{proposition}

\begin{proof}
First, consider $F\in L^{1}(\S ^{1})$ such that $F^{(N)}=f$ in the sense of
distributions. Then one may write that for any $\phi \in C^{\infty }(\S %
^{1}) $, 
\begin{equation*}
\langle \rho \ast f,\phi \rangle =\langle f,\bar\rho \ast \phi \rangle
=(-1)^{N}\langle F,(\bar\rho)^{(N)}\ast \phi \rangle =\langle \rho
^{(N)}\ast F,\phi \rangle,
\end{equation*}%
where we used $(\bar\rho)^{(N)}=(-1)^N\overline{\rho^{(N)}}$. Thus, the distribution $\rho \ast f$ coincides with the smooth function $%
\rho ^{(N)}\ast F$ in the sense of distributions: thus it is a
smooth function itself. \medskip For the second part, notice that one has,
for any $x\in \S ^{1}$, that 
\begin{equation*}
\rho \ast f(x)=\langle \delta _{x},\rho \ast f\rangle
\end{equation*}%
where $\delta _{x}$ is the Dirac mass at $x\in \S ^{1}$. Consider now $(\chi
_{n,x})_{n\geq 0}$ a mollifier, i.e a sequence of non-negative, smooth
functions with integral one 
such that for $g\in C^{\infty }(\S ^{1})$, $\langle \delta _{x},g\rangle
=\lim\limits_{n\rightarrow \infty }\langle \chi _{n,x},g\rangle $. \newline
In particular, we notice that for any $x\in \S ^{1}$ 
\begin{equation}
\langle \chi _{n,x},\rho \ast f\rangle =\int_{\S ^{1}}\chi _{n,x}(y)\rho
\ast f(y)dy=\langle \rho \ast f,\chi _{n,x}\rangle =\langle f,\bar \rho \ast \chi
_{n,x}\rangle
\end{equation}%
and thus 
\begin{equation}
|\langle \chi _{n,x},\rho \ast f\rangle |\leq \Vert f\Vert _{\mathcal{D}%
_{N}}\Vert \bar\rho \ast \chi _{n,x}\Vert _{C^{N}}\leq \Vert f\Vert
_{D_{N}}\Vert \rho \Vert _{C^{N}}\underset{=1}{\underbrace{\Vert \chi
_{n,x}\Vert _{L^{1}}}}.
\end{equation}%
Taking the limit $n\rightarrow +\infty $ gives the result for $k=0$. One
obtains the general case by replacing $\rho $ by $\rho ^{(k)}$ in the
previous computation.
\end{proof}

The previous discussion allows to give a precise meaning to the annealed
transfer operator $L_\delta$ \eqref{def:annealedtransferop}, and to its
derivative operators (see Definition \ref{def:derivop}).

\subsection{Small perturbations in the family of transfer operators}

\label{sec:smallpert} In this section, we establish the "small
perturbations" assumptions (LR4) and (QR2) of Theorems \ref{thm:linresp} and %
\ref{thm:quadresp}. \newline
We start by establishing that the perturbed transfer operator $L_{\delta }$
is close to $L_{0}$ in the $\Vert .\Vert _{L^{1}\rightarrow \mathcal{D}_{1}}$
norm, under the assumption that $D_{\delta }=Id+o_{\delta
\rightarrow 0}(1)$ in the $C^{0}$-topology, i.e if $\sup_{x\in \S %
^{1}}d(D_{\delta }(x),x)\underset{\delta \rightarrow 0}{\longrightarrow }0$.
This is in fact the consequence of the more general, following result:

\begin{proposition}
\label{prop:detpertisLip} Let $(T_{\delta })_{\delta \in \lbrack 0,\bar{%
\delta}]}$ be a family of continuous maps of the circle, such that $%
d_{C^{0}}(T_{\delta },T_{0})\underset{\delta \rightarrow 0}{\longrightarrow }%
0$ and consider their associated transfer operators $L_{T_{\delta }}$. Then 
\begin{equation*}
\Vert L_{T_{0}}-L_{T_{\delta }}\Vert _{L^{1}\rightarrow \mathcal{D}_{1}}%
\underset{\delta \rightarrow 0}{\longrightarrow }0.
\end{equation*}
\end{proposition}

\begin{proof}
First we consider functions $f\in L^{1}(\S ^{1})$ and $g\in C^{1}(\S ^{1})$.
One has, by duality properties of the transfer operator 
\begin{align*}
|\langle (L_{T_{0}}-L_{T_{\delta }})f,g\rangle |& =|\langle f,(g\circ
T_{\delta }-g\circ T_{0})\rangle | \\
& =\left\vert \int_{\S ^{1}}f(g\circ T_{\delta }-g\circ T_{0})dm\right\vert
\\
& \leq \Vert f\Vert _{L^{1 }}\sup_{x\in \S ^{1}}|g(T_{\delta
}(x))-g(T_{0}(x))| \\
& \leq \Vert f\Vert _{L^{1}}\Vert g\Vert _{C^{1}}d_{C^{0}}(T_{\delta },T_{0})
\end{align*}%
hence the result.
\end{proof}

We can then apply Proposition \ref{prop:detpertisLip} to the family $%
(D_\delta)_{\delta \in \lbrack 0,\overline{\delta }]}$, with $D_0=Id$%
. \medskip

Now we establish the derivative operator, or Taylor's expansion of order one
for the family of operators $(L_{D_{\delta }})_{\delta \in \lbrack 0,%
\overline{\delta }]}$. Assume that there exists $S\in C^{2}(\S ^{1},\mathbb{R%
})$, such that in the $C^{0}(\S ^{1})$-topology, $D_{\delta }=Id%
+\delta S+o(\delta )$, i.e 
\begin{equation}
\dfrac{1}{|\delta |}\Vert D_{\delta }-Id-\delta S\Vert _{C^{0}}%
\underset{\delta \rightarrow 0}{\longrightarrow }{0}  \label{eq:detpertisC1}
\end{equation}%
Note that as $S$ is a bounded function from $\S ^{1}$ to $\mathbb{R}$, the
product $f.S$ is well-defined for any $f\in L^{1}(\S ^{1})$; thus, we define 
$R=\left[ \dfrac{dL_{D_{\delta }}}{d\delta }\right] _{\delta =0}:L^{1}(\S %
^{1})\rightarrow W^{-1,1}(\S ^{1})\subset \mathcal{D}_{1}(\S ^{1})$ by 
\begin{equation}
Rf:=-(f.S)^{\prime }.
\end{equation}%
When $W^{-1,1}(\S ^{1})$ is endowed with the $\Vert .\Vert _{\mathcal{D}%
_{1}} $-topology, $R$ is a bounded operator. Beware that this is not the
derivative operator mentioned in Theorem \ref{thm:linresp} (which will be
associated to $\rho \ast L_{D_{\delta }}\circ L_{T}$ at $\delta =0$: see
Theorem \ref{thm:Taylorexpfortransferop}), but the derivative operator
associated to the family of diffeomorphisms $D_{\delta }$.

\begin{proposition}
\label{prop:transferopisC1} Let $(L_{D_{\delta }})_{\delta \in \lbrack 0,%
\overline{\delta }]}$ be the family of transfer operators associated to $%
(D_{\delta })_{\delta \in \lbrack 0,\overline{\delta }]}$. Then one has 
\begin{equation}
\left\Vert \dfrac{L_{D_{\delta }}-L_{D_{0}}}{\delta }-R\right\Vert
_{L^{1}\rightarrow \mathcal{D}_{2}}\underset{\delta \rightarrow 0}{%
\longrightarrow }0  \label{eq:dettransferopisC1}
\end{equation}
\end{proposition}

\begin{proof}
Let $f\in L^{\infty }(\S ^{1})$ and $g\in C^{2}(\S ^{1})$. One may write 
\begin{align*}
\left\vert \left\langle \dfrac{(L_{D_{\delta }}-L_{D_{0}})}{\delta }%
f-Rf,g\right\rangle \right\vert & =\left\vert \dfrac{1}{\delta }\langle
f,g\circ D_{\delta }-g-\delta g^{\prime }S\rangle \right\vert \\
& =\left\vert \dfrac{1}{\delta }\langle f,g\circ D_{\delta }-g-(D_{\delta }-%
Id)g^{\prime }+(D_{\delta }-Id-\delta S)g^{\prime }\rangle
\right\vert \\
& \leq \dfrac{1}{|\delta |}\int_{\S ^{1}}\left\vert f\right\vert dm\left[
\Vert g\circ D_{\delta }-g-(D_{\delta }-Id)g^{\prime }\Vert _{\infty
}+\Vert (D_{\delta }-Id-\delta S)g^{\prime }\Vert _{\infty }\right] .
\end{align*}%
One has, by the mean value theorem 
\begin{equation*}
\left\vert g(D_{\delta }(x))-g(x)-(D_{\delta }(x)-x)g^{\prime
}(x)\right\vert \leq \int_{x}^{D_{\delta }(x)}|g^{\prime }(t)-g^{\prime
}(x)|dt\leq C|\delta |\Vert g\Vert _{C^{2}}o_{\delta \rightarrow 0}(1)
\end{equation*}%
Together with the Taylor's expansion of $D_{\delta }$, this yields 
\begin{equation*}
\left\vert \left\langle \dfrac{(L_{D_{\delta }}-L_{D_{\delta }})}{\delta }%
f-Rf,g\right\rangle \right\vert \leq C\Vert f\Vert _{L^{1}}\Vert g\Vert
_{C^{2}}o_{\delta \rightarrow 0}(1)
\end{equation*}%
establishing \eqref{eq:dettransferopisC1}.
\end{proof}

Finally we show that a second order Taylor's expansion is satisfied. Assume
that there are $S_{1},~S_{2}\in C^{3}(\S ^{1},\mathbb{R})$ such that in the $%
C^{0}$-topology, $D_{\delta }$ satisfies 
\begin{equation*}
D_{\delta }=Id+\delta S_{1}+\dfrac{\delta ^{2}}{2}S_{2}+o(\delta
^{2})
\end{equation*}%
i.e 
\begin{equation}
\dfrac{1}{\delta ^{2}}\Vert D_{\delta }-Id-\delta S_{1}-\dfrac{%
\delta ^{2}}{2}S_{2}\Vert _{C^{0}}\underset{\delta \rightarrow 0}{%
\longrightarrow }0  \label{eq:detpertisC2}
\end{equation}%
and let us define the second derivative $Q:L^{1}(\S ^{1})\rightarrow
W^{-2,1}(\S ^{1})\subset \mathcal{D}_{2}(\S ^{1})$ by 
\begin{equation}
Qf:=(fS_{1}^{2})^{\prime \prime }-(fS_{2})^{\prime }.
\label{def:det2ndderivop}
\end{equation}%
When $W^{-2,1}(\S ^{1})$ is endowed with the topology induced by the $\Vert
.\Vert _{\mathcal{D}_{2}}$ norm, $Q$ is a bounded operator, and one has the following result.

\begin{proposition}
\label{prop:transferopisC2} Let $(L_{D_\delta})_{\delta\in[0,\overline{\delta%
}]}$ be the family of transfer operator associated to $(D_\delta)_{\delta\in[%
0,\overline{\delta}]}$. It satisfies 
\begin{equation}
\left\|\dfrac{L_{D_\delta}-L_0-\delta R}{\delta^2}-\dfrac{1}{2}%
Q\right\|_{L^1\to\mathcal{D}_3}\underset{\delta\rightarrow 0}{\longrightarrow%
}0.
\end{equation}
\end{proposition}

\begin{proof}
Let $f\in L^{1}(\S ^{1})$ and $g\in C^{3}(\S ^{1})$. Then one may write 
\begin{equation*}
\left\langle \dfrac{L_{D_{\delta }}-L_{0}-\delta R}{\delta ^{2}}f-\dfrac{1}{2%
}Qf,g\right\rangle =\dfrac{1}{\delta ^{2}}\int_{\S ^{1}}f(x)(g(D_{\delta
}(x))-g(x)-\delta Sg^{\prime }(x)-\dfrac{\delta ^{2}}{2}\left(
S_{1}^{2}g^{\prime \prime }+S_{2}g^{\prime }\right) (x))dx.
\end{equation*}%
Now, notice that one has: 
\begin{align*}
& g\circ D_{\delta }-g-\delta Sg^{\prime }-\dfrac{\delta ^{2}}{2}\left(
S_{1}^{2}g^{\prime \prime }+S_{2}g^{\prime }\right) \\
& =\underset{:=(I)}{\underbrace{g\circ D_{\delta }-g-(D_{\delta }-Id%
)g^{\prime }-\dfrac{1}{2}(D_{\delta }-Id)^{2}g^{\prime \prime }}}+%
\underset{:=(II)}{\underbrace{(D_{\delta }-Id-\delta S_{1}-\dfrac{%
\delta ^{2}}{2}S_{2})g^{\prime }+\dfrac{1}{2}\left( (D_{\delta }-Id%
)^{2}-\delta ^{2}S_{1}^{2}\right) g^{\prime \prime }}}.
\end{align*}%
It follows from Taylor's integral formula at order 3 and the Taylor's
expansion \eqref{eq:detpertisC2} that 
\begin{equation}
|(I)|\leq \Vert g\Vert _{C^{3}}o(\delta ^{3})
\end{equation}%
where the $o(\delta ^{3})$ is uniform in $x$. It also follows from %
\eqref{eq:detpertisC2} that $(D_{\delta }-Id)^{2}=\delta
^{2}S_{1}+o(\delta ^{2})$, (where once again, $o(\delta ^{2})$ is uniform in 
$x$) and thus 
\begin{equation}
|(II)|\leq \Vert g\Vert _{C^{2}}o(\delta ^{2}).
\end{equation}%
Finally, one obtains 
\begin{equation}
\left\vert \left\langle \dfrac{L_{D_{\delta }}-L_{0}-\delta R}{\delta ^{2}}%
f-Qf,g\right\rangle \right\vert \leq \Vert f\Vert _{L^{1}}\Vert g\Vert
_{C^{3}}o_{\delta \rightarrow 0}(1)
\end{equation}%
hence the result.
\end{proof}

We now consider the \textbf{derivative operators} of the system with
additive noise, defining $\dot{L }$ and $\ddot{L }$.

\begin{definition}
\label{def:derivop} Let $(L_{\delta })_{\delta \in \lbrack -\epsilon
,\epsilon ]}$ be the family of transfer operators associated with %
\eqref{eq:syswaddnoise}. The \emph{derivative operators} $\dot{L}:C^{k}(\S %
^{1})\rightarrow C^{k}(\S ^{1})$ and $\ddot{L}:C^{k}(\S ^{1})\rightarrow
C^{k}(\S ^{1})$ are defined by: 
\begin{align}
\dot{L}& :=\rho _{\xi }\ast R\circ L_{T}  \label{eq:derivop} \\
\ddot{L}& :=\rho _{\xi }\ast Q\circ L_{T}.
\end{align}
\end{definition}

\begin{remark}
The convolution in \eqref{eq:derivop} should be understood in the sense of
Definition \ref{def:convprodII}. Notice that the regularization effect of
the Gaussian noise allow us to define the derivative operators $\dot{L}%
:C^{\infty }(\S ^{1})\rightarrow C^{\infty }(\S ^{1})$, and even from $L^{1}(%
\S ^{1})$ to $C^{\infty }(\S ^{1})$ (see Proposition \ref{prop:regineqII}
and \ref{prop:transferopisregularizing}).
\end{remark}

\begin{theorem}
\label{thm:Taylorexpfortransferop} Let $(L_{\delta })_{\delta \in \lbrack 0,%
\overline{\delta }]}$ be the family of transfer operators associated\ to
systems of the kind described in \eqref{eq:syswaddnoise} and perturbations
satisfying $($\ref{eq:detpertisC2}$)$. \newline
Then for any $k\in \mathbb{N}$, the derivative operators $\dot{L}:C^{k+1}(\S %
^{1})\rightarrow C^{k}(\S ^{1})$ and $\ddot{L}:C^{k+1}(\S ^{1})\rightarrow
C^{k-1}(\S ^{1})$ satisfy the following estimates: 
\begin{align}
& \Vert L_{0}-L_{\delta }\Vert _{C^{k+1}\rightarrow C^{k}}\leq C\delta ; \\
& \left\Vert \dfrac{L_{\delta }-L_{0}}{\delta }-\dot{L}\right\Vert
_{C^{k}\rightarrow C^{k-1}}\underset{\delta \rightarrow 0}{\longrightarrow }%
0; \\
& \left\Vert \dfrac{L_{\delta }-L_{0}-\delta \dot{L}}{\delta ^{2}}-\dfrac{1}{%
2}\ddot{L}\right\Vert _{C^{k}\rightarrow C^{k-2}}\underset{\delta
\rightarrow 0}{\longrightarrow }0.
\end{align}
\end{theorem}

\begin{proof}
Recall that by \eqref{def:annealedtransferop}, one has $L_{\delta }=\rho
_{\xi }\ast L_{D_{\delta }}$. Let $\phi \in C^{k+1}(\S ^{1})$. 
Applying Proposition \ref{prop:detpertisLip} and \eqref{eq:regineqI} yields 
\begin{equation}
\Vert (L_{0}-L_{\delta })\phi \Vert _{C^{k}}\leq C\delta \Vert \rho \Vert
_{C^{k}}\Vert L_{T}\phi \Vert _{L^{1}}\leq C\delta \Vert \phi \Vert
_{C^{k+1}}.
\end{equation}%
Take $\phi \in C^{k}(\S ^{1})$. By Proposition \ref{prop:transferopisC1} and %
\eqref{eq:regineqII} one has 
\begin{align}
\left\Vert \dfrac{L_{\delta }-L_{0}}{\delta }\phi -\dot{L}\phi \right\Vert
_{C^{k-1}}& \leq C\Vert \rho \Vert _{C^{k+1}}\left\Vert \dfrac{L_{D_{\delta
}}-Id}{\delta }L_{T}\phi -RL_{T}\phi \right\Vert _{\mathcal{D}_{2}} \\
& \leq C\Vert \rho \Vert _{C^{k+1}}\Vert L_{T}\phi \Vert _{L^{1}}\left\Vert 
\dfrac{L_{D_{\delta }}-Id}{\delta }-R\right\Vert _{L^{1}\rightarrow \mathcal{%
D}_{2}}.
\end{align}%
Similarly combining Proposition \ref{prop:transferopisC2} and %
\eqref{eq:regineqII}, one obtains 
\begin{equation}
\left\Vert \dfrac{L_{\delta }-L_{0}-\delta \dot{L}}{\delta ^{2}}\phi -\dfrac{%
1}{2}\ddot{L}\phi \right\Vert _{C^{k-2}}\leq C\Vert \rho \Vert
_{C^{k+1}}\Vert L_{T}\phi \Vert _{L^{1}}\left\Vert \dfrac{L_{D_{\delta
}}-Id-\delta R}{\delta ^{2}}-\dfrac{1}{2}Q\right\Vert _{L^{1}\rightarrow 
\mathcal{D}_{3}}.
\end{equation}%
Hence the result.
\end{proof}

\subsection{Convergence to equilibrium and regularization for the unperturbed transfer operator}

In this subsection, we show that the unperturbed transfer operator $L
_{0}:=\rho _{\xi }\ast L_{T}$ satisfies the rest of the assumptions of
Theorems \ref{thm:linresp} and \ref{thm:quadresp}, with the nested sequence
of Banach spaces $C^{k+1}(\S ^{1})\hookrightarrow C^{k}(\S %
^{1})\hookrightarrow C^{k-1}(\S ^{1})\hookrightarrow C^{k-2}(\S ^{1})$, via
the result of section \ref{resolvsec}, namely the convergence to equilibrium on $C^{k}(\S ^{1})$,
weak boundedness for the sequence $(\Vert L _{0}^{n}\Vert _{C^{k}})_{n\in 
\mathbb{N}}$ and regularizing from $L^{1}(\S ^{1})$ to $C^{k}(\S ^{1})$ for
any $k\in \mathbb{N}$. Notice that here we use crucially the weak
contraction property of the (deterministic) transfer operator $L_{T} $ on $%
L^{1}(\S ^{1})$ as a preliminary to the subtler regularization properties.

\begin{lemma}
\label{prop:transferopisregularizing} Let $k\geq 0$.

\begin{enumerate}
\item The unperturbed transfer operator $L _0$ is regularizing from $L^1(\S %
^1)$ to $C^k(\S ^1)$ and from $C^k(\S ^1)$ to $C^{k+1}(\S ^1)$ for any $k\in%
\mathbb{N}$.

\item The unperturbed transfer operator $L_{0}$ has convergence to equilibrium on $C^{k+1}(\S %
^{1})$, i.e there is $a_n\to 0$ such that for any $g\in C^{k+1}(\S ^{1})$, such that $\int_{\S ^{1}}gdm=0$%
, then 
\begin{equation*}
\Vert L_{0}^{n}g\Vert _{C^{k}} \leq a_n \Vert g\Vert _{C^{k+1}}.
\end{equation*}

\item The sequence $(\Vert L _{0}^{n}\Vert _{C^{k}})_{n\in \mathbb{N}}$ is
bounded, i.e there exists $M^{\prime }>0$ such that 
\begin{equation*}
\Vert L _{0}^{n}\phi \Vert _{C^{k}}\leq M^{\prime }\Vert \phi \Vert _{C^{k}}
\end{equation*}%
for all $\phi \in C^{k}(\S ^{1})$.
\end{enumerate}
\end{lemma}

\begin{proof}
The regularization property from $C^{k}(\S ^{1})$ to $C^{k+1}(\S ^{1})$ is a
straightforward consequence of the regularization inequalities %
\eqref{eq:regineqII} for $N=0$. \newline
For the regularization property from $L^{1}(\S ^{1})$ to $C^{k}(\S ^{1})$,
one may see that any $f\in L^{1}(\S ^{1})$ admits an anti derivative (in the
sense of distributions) $F\in W^{1,1}(\S ^{1})$, so that in fact $f\in
W^{-1,1}(\S ^{1})$. Thus it follows that Proposition \ref{prop:regineqII}
applies, so that $L _{0}f\in C^{\infty }(\S ^{1})$ and, since the injection $%
L^1\to\mathcal{D}_0$ is continuous: 
\begin{equation*}
\Vert L _{0}f\Vert _{C^{k}}\leq \Vert \rho _{\xi }\Vert _{C^{k}}\Vert
L_{T}f\Vert _{\mathcal{D}_{0}}\leq \Vert \rho _{\xi }\Vert _{C^{k}}\Vert
L_{T}f\Vert _{L^{1}}\leq \Vert \rho _{\xi }\Vert _{C^{k}}\Vert f\Vert
_{L^{1}}.
\end{equation*}%
Hence $L _{0}:L^{1}(\S ^{1})\rightarrow C^{k}(\S ^{1})$ is continuous for
any $k\in \mathbb{N}$. \newline
For the second item, one may remark that $L_{0}$ has a stictly  positive kernel. There is $l>0$ such that the kernel $k(x,y)$ of $L_0$ satisfies  $k(x,y)\geq l$ and
thus (see note  6  of  \cite{LSV} or the proof of \cite[Corollary 5.7.1]{LM}) we have that for any $g\in V_{k+1}(\S ^{1})$,
$\Vert L_{0}^{n}g\Vert _{L^{1}}\leq  (1-l)^{n} ||g||_{L^1}$. Thus by the regularization property, one gets 
\begin{equation*}
\Vert L_{0}^{n}g\Vert _{C^{k}}\leq C\Vert L_{0}^{n-1}g\Vert _{L^{1}} \leq C (1-l)^{n-1} ||g||_{L^1}\leq C (1-l)^{n-1} ||g||_{C^{k+1}}
\end{equation*}
which proves the claim.
For the last item, we once again use the regularization property from $L^{1}(%
\S ^{1})$ to $C^{k}(\S ^{1})$, as such. First, we start by remarking that
for any $f\in L^{1}(\S ^{1})$, the convolution product defined in %
\eqref{def:convprod} has the following property: the function $\rho _{\xi
}\ast f\in L^{1}(\S ^{1})$, and 
\begin{equation*}
\Vert \rho _{\xi }\ast f\Vert _{L^{1}(\S ^{1})}\leq \Vert \rho _{\xi }\Vert
_{L^{1}(\mathbb{R})}\Vert f\Vert _{L^{1}(\S ^{1})}=\Vert f\Vert _{L^{1}(\S %
^{1})}
\end{equation*}%
since $\rho _{\xi }$ is a probability kernel. Hence, one has, for any $\phi
\in C^{k}(\S ^{1})\subset L^{1}(\S ^{1})$, 
\begin{equation*}
\Vert L _{0}\phi \Vert _{L^{1}}\leq \Vert L_{T}\phi \Vert _{L^{1}}\leq \Vert
\phi \Vert _{L^{1}}
\end{equation*}%
which gives, by an immediate induction, $\Vert L _{0}^{N}\phi \Vert
_{L^{1}}\leq \Vert \phi \Vert _{L^{1}}$ for any $N\in \mathbb{N}$ and any $%
\phi \in C^{k}(\S ^{1})$. Thus 
\begin{equation*}
\Vert L _{0}^{n}\phi \Vert _{C^{k}}\leq \Vert L _{0}\Vert _{L^{1}\rightarrow
C^{k}}\Vert L _{0}^{n-1}\phi \Vert _{L^{1}}\leq \Vert L _{0}\Vert
_{L^{1}\rightarrow C^{k}}\Vert \phi \Vert _{L^{1}}\leq \Vert L _{0}\Vert
_{L^{1}\rightarrow C^{k}}\Vert \phi \Vert _{C^{k}}
\end{equation*}
\end{proof}

\begin{remark}
The previous Lemma also applies to each of the perturbed operators $L_\delta$, 
$\delta\in[0,\bar\delta]$.
\end{remark}

We may summarize the conclusions of Section \ref{secnoise} in the following
way:

\begin{theorem}
Let $T:\S ^1\to \S ^1$ be a non-singular map and $(D_\delta)_{\delta\in[0,%
\bar{\delta}]}$ a family of diffeomorphisms of the circle, satisfying %
\eqref{eq:detpertisC1} and \eqref{eq:detpertisC2}. \newline
We consider the random dynamical system \eqref{eq:syswaddnoise} generated by 
\begin{equation}
T_\delta(\omega,x)=D_\delta\circ T(x)+X_\xi(\omega) \mod 1
\end{equation}
where $X_\xi$ is a centered Gaussian random variable with variance $\xi^2$,
and the associated (annealed) transfer operator $(L_\delta)_{\delta\in[0,%
\bar{\delta}]}$ defined by \eqref{def:annealedtransferop}. Then Theorem \ref%
{thm:linresp} and \ref{thm:quadresp} apply for the sequence of spaces $%
C^{k+1}(\S ^1)$, $C^{k}(\S ^1)$, $C^{k-1}(\S ^1)$ and $C^{k-2}(\S ^1)$. i.e
linear and quadratic response hold for the stationary measure when $%
\delta\to 0$.
\end{theorem}

\begin{proof}
By Lemma \ref{prop:transferopisregularizing} and the remark right after,
such a system satisfies the assumptions of Theorem \ref{gap} for
(say) $B_w=C^k(\S ^1)$ and $B_s=C^{k+1}(\S ^1)$. Hence we get the existence
and boundedness of the stationary densities (Assumption LR1) $(h_{\delta
})_{\delta \in [0,\bar{\delta}]}$, as well as the good definition of the
resolvent operator $R(1,L_{0})$ on the spaces $C^k(\S ^1)$ (Assumptions LR3,
QR3). \newline
The convergence to equilibrium property for the unperturbed operator (Assumption LR2) is
directly established in Lemma \ref{prop:transferopisregularizing}. \newline
The regularity properties for the family of transfer operators (Assumptions
LR4, QR1, QR2) are established in Theorem \ref{thm:Taylorexpfortransferop},
Section \ref{sec:smallpert} for the nested sequence of spaces $C^{k+1}(\S %
^1)\subset C^{k}(\S ^1)\subset C^{k-1}(\S ^1)\subset C^{k-2}(\S ^1)$.
\end{proof}

\begin{remark}
One may adapt the estimates in the proof of Theorem \ref%
{thm:Taylorexpfortransferop} so that assumptions of Theorems \ref%
{thm:linresp} and \ref{thm:quadresp} hold for the spaces $%
B_{ss}=B_s=B_w=B_{ww}=C^k(\S ^1)$. Notice that in this case, the assumptions of Theorem \ref{gap} are still satisfied (see Remark \ref{rem:compactness}).
\end{remark}

\subsection{Application: Arnold maps with Gaussian noise}

In this subsection we present an example to which the previous approach
apply: the Arnold standard map of the circle, perturbed with Gaussian noise. 
\newline
More precisely, one takes $D_{\delta }:=Id+\delta $ to be the
rotation of angle $\delta $, and $T$ to be the standard Arnold circle map 
\begin{equation*}
T(x):=x+a+\epsilon \sin (2\pi x)\mod 1
\end{equation*}%
with $\epsilon >0$: in particular, it does not matter to us whether $T$ is a
diffeomorphism ($\epsilon <1$) or not $(\epsilon >1)$. Then the random
dynamical system induced by this data and a sequence of i.i.d Gaussian
random variable $(\Omega _{n})_{n\geq 0}$,%
\begin{equation*}
X_{n+1}=D_{\delta }\circ T(X_{n})+\Omega _{n}
\end{equation*}%
satisfies the assumptions of Section \ref{secnoise}, for the sequence of
spaces $C^{k+1}(\S ^{1})\subset C^{k}(\S ^{1})\subset C^{k-1}(\S %
^{1})\subset C^{k-2}(\S ^{1})$; in particular linear response holds if we
see the density of the stationary measure $h_{\delta }\in C^{k-1}(\S ^{1})$
and quadratic response holds if we consider $h_{\delta }\in C^{k-2}(\S ^{1})$%
. \newline
It is also possible to proceed as in \cite{MSGDG} (Proposition 17) and write
the (almost surely constant) rotation number of this random dynamical system
as the integral of some well-chosen observable against its stationary
measure, and thus deduce its regularity w.r.t the "driving frequency" $a$
(Corollary 18).

\section{Linear and Quadratic response for expanding maps\label{secmap}}

In this section we consider smooth expanding maps on the circle and show
they have linear and quadratic response with respect to smooth
perturbations. We also provide explicit formulas for the response.

To get the linear response we will consider maps $T:\S^{1}\rightarrow \S^{1}$
satisfying the following assumptions

\begin{enumerate}
\item $T\in C^{4},$

\item $|T^{\prime }(x)|\geq\alpha^{-1}>1$ $\forall x$.
\end{enumerate}

For the quadratic response we will consider $T\in C^{5}.$ We consider a
family of perturbations of $T:=T_{0}$ of the kind $T_{\delta }:=D_{\delta
}\circ T$ with $D_{\delta }=Id+o_{\delta \rightarrow 0}(1)$ in a suitable topology.

In the following subsection we show these systems satisfy the assumptions of
Theorems \ref{th:linearresponse} and \ref{thm:quadresp}.

\subsection{Resolvent for expanding maps}

In this section we show the existence and continuity properties of the
resolvent, needed to apply our response statements to deterministic
expanding maps, via Section \ref{resolvsec}.

More precisely, we show that the transfer operators associated to expanding
maps satisfy regularization (here Lasota-Yorke) inequalities (see Assumption 
$1$ of Theorem \ref{gap}) when acting on suitable Sobolev spaces.

\begin{lemma}
\label{Lemsu}A $C^{k+1}$ expanding map on $\S^{1}$ satisfies a Lasota-Yorke
inequality on $W^{k,1}(\S ^{1})$: there is $\alpha <1$, $A_{k},~B_{k}\geq 0$
such that%
\begin{equation*}
\left\{ \begin{aligned} &\|L^nf\|_{W^{k-1,1}}\leq A_k\|f\|_{W^{k-1,1}}\\
&\|L^{n}f\|_{W^{k,1}}\leq \alpha^{kn}\|f\|_{W^{k,1}}+B_k\|f\|_{W^{k-1,1}}.
\end{aligned}\right. .
\end{equation*}
\end{lemma}

\begin{proof}
We will proceed by induction on $k\geq 1$. The Lasota-Yorke inequality for $%
k=1$ is well-known, and we refer to, e.g. \cite[Section 4.1]{notes} for details. \newline
Assume that the Lasota-Yorke inequality is established on $W^{k-1,1}(\S %
^{1}) $ for the transfer operator of a $C^{k}$ expanding map, and consider a 
$C^{k+1}$ expanding map together with its associated transfer operator. Let $%
f\in W^{k,1}(\S ^{1})$. We want to evaluate 
\begin{equation*}
\Vert Lf\Vert _{W^{k,1}}=\Vert (Lf)^{\prime }\Vert _{W^{k-1,1}}+\Vert
Lf\Vert _{L^{1}}.
\end{equation*}%
By the well-known formula, we have that $(Lf)^{\prime }=L\left( \dfrac{%
f^{\prime }}{T^{\prime }}\right) +L\left( \dfrac{T^{\prime \prime }}{%
(T^{\prime })^2}f\right) $, thus we may write, using our induction hypothesis 
\begin{equation*}
\Vert (Lf)^{\prime }\Vert _{W^{k-1,1}}\leq \alpha ^{k-1}\left\Vert \dfrac{%
f^{\prime }}{T^{\prime }}\right\Vert _{W^{k-1,1}}+C_{k}^{\prime }\Vert
f\Vert _{W^{k-1,1}}.
\end{equation*}%
For our purposes, we only need precise information on the term carrying the
highest derivative of $f$: by Leibniz formula, one has 
\begin{equation*}
\left( \dfrac{f^{\prime }}{T^{\prime }}\right) ^{(k-1)}=\sum_{\ell
=0}^{k-1}\left( \dfrac{1}{T^{\prime }}\right) ^{k-1-\ell }f^{(\ell +1)}=%
\dfrac{f^{(k)}}{T^{\prime }}+\dots
\end{equation*}%
so that one gets $\left\Vert \dfrac{f^{\prime }}{T^{\prime }}\right\Vert
_{W^{k-1,1}}\leq \left\Vert \dfrac{1}{T^{\prime }}\right\Vert _{\infty
}\Vert f^{(k)}\Vert _{L^{1}}+C_{k}^{\prime \prime }\Vert f\Vert _{W^{k-1,1}}$%
, whence 
\begin{equation}
\Vert Lf\Vert _{W^{k,1}}\leq \alpha ^{k}\Vert f\Vert _{W^{k,1}}+C_{k}\Vert
f\Vert _{W^{k-1,1}}.
\end{equation}%
We may now iterate this inequality; after $n$ steps we obtain 
\begin{equation*}
\Vert L^{n}f\Vert _{W^{k,1}}\leq \alpha ^{nk}\Vert f\Vert
_{W^{k,1}}+C_{k}\sum_{i=0}^{n-1}\alpha ^{k(n-1-i)}\Vert L^{i}f\Vert
_{W^{k-1,1}}
\end{equation*}%
and the wanted result follows by power-boundedness of $L^{i}$ on $W^{k-1,1}$%
, with \newline
$B_{k}:=\dfrac{A_{k}C_{k}}{1-\alpha ^{k}}$.
\end{proof}

From this last result, we classically deduce the following: for any $k\geq 1$%
, the transfer operator $L_T$ of a $C^{k+1}$ expanding map $T$ is \emph{%
quasi-compact} on $W^{k,1}(\S ^1)$. Furthermore, by topological
transitivity, $1$ is the only eigenvalue on the unit circle. It is simple
and the associated (normalized) eigenfunction, $h$, is the invariant density
of the system. The rest of the spectrum is contained in disk of radius
strictly smaller than one. In particular, we may write $L=\Pi+R$ with $%
\Pi:W^{k,1}(\S ^1)\to W^{k,1}(\S ^1)$ the spectral projector, defined by $%
\Pi(\phi):=h\int_{\S ^1}\phi dm$ and $R$ satisfying $R\Pi=\Pi R=0$ and $%
\|R^n\phi\|_{W^{k,1}}\leq C\rho^n\|\phi\|_{W^{k,1}}$. Thus we have the following result.

\begin{proposition}
\label{propora}For each $g\in V_k:=\{g\in W^{k,1}(\S ^1)~s.t.~\int_{\S ^1}
g~dm=0\}=\ker(\Pi)$, it holds 
\begin{equation*}
\|L^{n}g\|_{W^{k,1}}\leq C\rho^n\|g\|_{W^{k,1}}.
\end{equation*}
In particular, the resolvent $R(1,L):=(Id-L)^{-1}=\sum_{i=0}^\infty L^i$ is
a well-defined and bounded operator on $V_k$.
\end{proposition}

\subsection{Small perturbations of expanding maps}

\label{sec:pertofexpmaps}

In this section, we specify the type of perturbations we consider in the
deterministic case, and establish that they satisfy the relative continuity,
and Taylor's expansions assumptions (LR4), (QR2) for the spaces $W^{4,1}(\S %
^{1}),W^{3,1}(\S ^{1}),W^{2,1}(\S ^{1}),W^{1,1}(\S ^{1})$. We will focus on
the case of a fixed, $C^{4}$ expanding map of the circle $T:\S %
^{1}\rightarrow \S ^{1}$, perturbed by left composition with a family of
diffeomorphisms $(D_{\delta })_{\delta \in \lbrack -\epsilon ,\epsilon ]}$.
\medskip

More precisely, let $D_{\delta }:\S ^{1}\rightarrow \S ^{1}$ be a
diffeomorphism, with 
\begin{equation}  \label{def:diffeoperturb}
D_\delta=Id+\delta S
\end{equation}
and $S\in C^{k+1}(\mathbb{S}^1,\mathbb{R})$. For such a diffeomorphism, one
has the follwing result.

\begin{lemma}
Assume $k=0$, i.e. $S\in C^1(\S^1)$. Then in the $C^{0}(\S ^{1})$-topology 
\begin{equation}
\left\Vert \dfrac{1}{\delta }(D_{\delta }^{-1}-Id)+S\right\Vert _{C^{0}(\S %
^{1})}\underset{\delta \rightarrow 0}{\longrightarrow }0
\label{eq:1storderTaylordiffeo}
\end{equation}%
which we sum up in 
\begin{equation*}
D_{\delta }^{-1}=Id-\delta S+o(\delta ).
\end{equation*}
\end{lemma}

\begin{remark}\label{rem:1storderTaylordiffeo}
The $o(\delta)$ must be understood as a $C^0$ function that goes to zero
with $\delta$, uniformly in $x$. 
\\Although we do not prove it here, note that in the general case $S\in C^{k+1}(\S^1)$, the result holds in $C^k(\S^1)$-topology, which means the $o(\delta)$ is a $C^k$ function that goes to $0$ as $\delta\to 0$, as well as its derivatives.
\end{remark}

\begin{proof}
	Let $x\in\S^1$, and $y=D_\delta(x)$. Then
	\begin{align*}
	|D_\delta^{-1}(y)-y+\delta S(y)|=|x-D_\delta(x)+\delta S(D_\delta(x))|&=\delta|S(D_\delta(x))-S(x)| \\
	&\leq \delta\|S'\|_\infty d(D_\delta(x),x)\\
	&\leq \delta^2\|S'\|_\infty\|S\|_\infty
	\end{align*}
	hence the result.
\end{proof}

Let $k\in \mathbb{N}$. Our starting point is the remark that for any map $%
g\in C^{k}(\S ^{1})$, the operator $M_{g}$ defined by $M_{g}(f):=g.f$ is
bounded on $W^{k,1}(\S ^{1})$: this is an easy consequence of Leibniz
formula. In turns, this implies the following proposition.

\begin{proposition}
\label{prop:transferopisbounded} The transfer operator $L_{D_\delta}$
associated to a $C^{k+1}$-diffeomorphism $D_\delta$ is bounded on $W^{k,1}(%
\S ^1)$.
\end{proposition}
We introduce the notation 
\begin{equation}  \label{eq:weight}
J_\delta:=\dfrac{1}{1+\delta S^{\prime }\circ D_\delta^{-1}}
\end{equation}
for the weight of $L_{D_\delta}$. We prove the proposition by induction on $%
k\in\mathbb{N}$.
\begin{proof}
For $k=0$, the claim is simply that $L_{D_\delta}$ is bounded on $L^1(\S ^1)$, 
which is well-known. \newline
Let us assume that $L_{D_\delta}:W^{k-1,1}(\S ^1)\circlearrowleft$ is a
bounded operator whenever $D_\delta=Id+\delta.S$ with $S\in C^{k}(\S %
^1)$. If $S\in C^{k+1}(\S ^1)$ we write, for $f\in W^{k,1}(\S ^1)$ that 
\begin{equation*}
\|L_{D_\delta}f\|_{W^{k,1}}=\|(L_{D_\delta}f)^{\prime
}\|_{W^{k-1,1}}+\|L_{D_\delta}f\|_{L^1}.
\end{equation*}
We can thus write that $(L_{D_\delta}f)^{\prime }=J^{\prime }_\delta f\circ
D_{\delta}^{-1}+J_\delta^2 f^{\prime }\circ D_\delta^{-1}$. As 
\begin{equation}  \label{eq:weightdiff}
J^{\prime }_\delta=-\dfrac{\delta S^{\prime \prime }\circ D_{\delta}^{-1}}{%
(1+\delta S^{\prime }\circ D_\delta^{-1})^2}J_\delta,
\end{equation}
one has by using the remark above the statement of Proposition \ref{prop:transferopisbounded}
\begin{equation*}
\|(L_{D_\delta}f)^{\prime }\|_{W^{k-1,1}}\leq
C_\delta\|L_{D_\delta}f\|_{W^{k-1,1}}+ C^{\prime
}_\delta\|L_{D_\delta}f^{\prime }\|_{W^{k-1},1}.
\end{equation*}
The result follow by induction hypothesis.
\end{proof}

We now turn to continuity estimates for the map $\delta \mapsto L_{D_{\delta
}}\in L(W^{k,1}(\S ^{1}),W^{k-1,1}(\S ^{1}))$. Our first step is the
following lemma.

\begin{lemma}
Let $f\in W^{1,1}(\S ^{1})$, and let $H:\S ^{1}\rightarrow \S ^{1}$ be an
orientation-preserving homeomorphism. Then 
\begin{equation}
\Vert f\circ H-f\Vert _{L^{1}}\leq \Vert H^{-1}-Id\Vert _{\infty
}\Vert f^{\prime }\Vert _{L^{1}}.
\end{equation}
\end{lemma}

\begin{proof}
This is a straightforward consequence of Fubini-Tonelli theorem. Lifting
everything to $\mathbb{R}$, we consider a monotone, increasing lift of $H$
that we still denote $H$. Since $f$ is $W^{1,1}$, it is the integral of its
derivative and one has 
\begin{align*}
\int_0^1 (f\circ H(x)-f(x))dx&=\int_0^1\int_0^1 1_{[x,H(x)]}f^{%
\prime }(t)dtdx=\int_0^1f^{\prime }(t)\int_0^1 1%
_{[H^{-1}(t),t]}(x)dxdt.
\end{align*}
The result follow by taking absolute values in the previous equality.
\end{proof}

\begin{proposition}
\label{prop:strongcontinuity} For an orientation preserving, $C^{k+1}$
diffeomorphism $D_{\delta }=Id+\delta .S$, one has 
\begin{equation}
\Vert L_{D_{\delta }}-Id\Vert _{W^{k,1}\rightarrow W^{k-1,1}}\leq C\delta .
\label{eq:strongcontinuity}
\end{equation}
\end{proposition}

\begin{proof}
We prove the proposition by induction on $k\in \mathbb{N}$. For $k=1$, let $%
f\in W^{1,1}(\S ^{1})$ and write $L_{D_{\delta }}f-f=(J_{\delta }-1)f\circ
D_{\delta }^{-1}+f\circ D_{\delta }^{-1}-f$, so that 
\begin{align*}
\Vert L_{D_{\delta }}f-f\Vert _{L^{1}}& \leq \Vert \delta S^{\prime }\circ
D_{\delta }^{-1}L_{D_{\delta }}f\Vert _{L^{1}}+\Vert f\circ D_{\delta
}^{-1}-f\Vert _{L^{1}} \\
& \leq \delta \Vert S^{\prime }\Vert _{\infty }\Vert f\Vert _{L^{1}}+\delta
\Vert S\Vert _{\infty }\Vert f^{\prime }\Vert _{L^{1}} \\
& \leq \Vert S\Vert _{C^{1}}\delta \Vert f\Vert _{W^{1,1}},
\end{align*}%
by using the previous lemma and the remark above the statement of Proposition \ref{prop:transferopisbounded}. Let us now assume
that the proposition holds at rank $k$, and let $f\in W^{k+1,1}(\S ^{1})$. 
\newline
One has $\Vert L_{D_{\delta }}f-f\Vert _{W^{k,1}}=\Vert (L_{D_{\delta
}}f-f)^{\prime }\Vert _{W^{k-1,1}}+\Vert L_{D_{\delta }}f-f\Vert _{L^{1}}$,
with 
\begin{align*}
(L_{D_{\delta }}f-f)^{\prime }& =J_{\delta }^{\prime }f\circ D_{\delta
}^{-1}+J_{\delta }^{2}f^{\prime }\circ D_{\delta }^{-1}-f^{\prime } \\
& =J_{\delta }^{\prime }f\circ D_{\delta }^{-1}+J_{\delta }(L_{D_{\delta
}}(f^{\prime })-f^{\prime })+(J_{\delta }-1)f^{\prime }.
\end{align*}%
In view of \eqref{eq:weight},\eqref{eq:weightdiff} and remark \ref{rem:1storderTaylordiffeo} $(J_{\delta }-1)/\delta
J_{\delta }$ and $J_{\delta }^{\prime}/\delta J_{\delta }$ are bounded in $\|.\|_{C^{k}}$-norm when $\delta\to 0$.
\\By induction
hypothesis, $\Vert L_{D_{\delta }}(f^{\prime })-f^{\prime }\Vert
_{W^{k-1,1}}\leq C\delta \Vert f^{\prime }\Vert _{W^{k,1}}$. Thus, 
\begin{equation*}
\Vert (L_{D_{\delta }}f-f)^{\prime }\Vert _{W^{k-1,1}}\leq C\delta \Vert
L_{D_{\delta }}f\Vert _{W^{k-1,1}}+C\delta \Vert f^{\prime }\Vert
_{W^{k,1}}+C\delta \Vert L_{D_{\delta }}f^{\prime }\Vert _{W^{k-1,1}},
\end{equation*}%
and the conclusion follows from the case $k=1$ and Proposition \ref%
{prop:transferopisbounded}.
\end{proof}

We now turn to differentiability estimates, i.e we establish first-order
Taylor's expansion for the map $\delta \mapsto L_{D_{\delta }}\in L(W^{k,1}(%
\S ^{1}),W^{k-1,1}(\S ^{1}))$. \newline
In this endeavor, the first step is to define the \emph{derivative operator} 
$R:W^{k,1}(\S ^{1})\rightarrow W^{k-1,1}(\S ^{1})$, by 
\begin{equation}
R(f):=-(fS)^{\prime }.  \label{def:detderivop}
\end{equation}%
Then we have the following proposition.

\begin{proposition}
\label{prop:1storderTaylordetop} Consider an orientation-preserving, $%
C^{k+1} $ diffeomorphism $D_{\delta }:\S ^{1}\rightarrow \S ^{1}$ as in %
\eqref{def:diffeoperturb}. \newline
Let us consider the associated transfer operator $L_{D_\delta}:W^{k,1}(\S %
^{1})\rightarrow W^{k,1}(\S ^{1})$. Then 
\begin{equation}  \label{eq:1storderTaylorop}
\lim_{\delta \rightarrow 0}\left\|\frac{L_{D_\delta}f-f}{\delta}
+(fS)^{\prime }\right\|_{W^{k-1,1}}=0
\end{equation}
for each $f\in W^{k,1}(\S ^{1})$.
\end{proposition}

\begin{proof}
We start by writing, for $f\in W^{k,1}(\S ^1)$, 
\begin{align*}
L_{D_\delta}f-f-\delta R(f)&=J_\delta f\circ
D_\delta^{-1}-f+\delta(fS)^{\prime } \\
&=\underset{:=(I)}{\underbrace{(J_\delta-1)f\circ D_\delta^{-1}+\delta
f.S^{\prime }}}+\underset{:=(II)}{\underbrace{f\circ D_\delta^{-1}-f+(\delta
S)f^{\prime }}}.
\end{align*}
As $(I)=-\delta\left(L_{D_\delta}(fS^{\prime })-fS^{\prime }\right)$, by
Proposition \ref{prop:strongcontinuity} we have the bound 
\begin{equation*}
\|(I)\|_{W^{k-1,1}}\leq C\delta^2\|f\|_{W^{k,1}}
\end{equation*}
for all $k\in\mathbb{N}$, which obviously implies the wanted result. For $%
(II)$ we proceed by induction on $k\in\mathbb{N}$. First we consider, for $%
f\in W^{1,1}(\S ^1)$, 
\begin{align*}
\dfrac{1}{\delta}(f(D_\delta^{-1}(x))-f(x)+\delta Sf^{\prime }(x))&=\dfrac{1%
}{\delta}\int_x^{D_\delta^{-1}x}(f^{\prime }(t)-f^{\prime }(x))dt+\dfrac{1}{%
\delta}(D_\delta^{-1}(x)-x+\delta S(x))f^{\prime }(x) \\
&=\dfrac{D_\delta^{-1}(x)-x}{\delta}.\dfrac{1}{D_\delta^{-1}(x)-x}%
\int_x^{D_\delta^{-1}x}(f^{\prime }(t)-f^{\prime }(x))dt+o(1)f^{\prime }(x)
\\
&\underset{\delta\rightarrow 0}{\longrightarrow}0,
\end{align*}
for a.e $x\in\S ^1$ by Lebesgue's differentiation theorem. Note that the $%
o(1)$ comes from \eqref{eq:1storderTaylordiffeo}, and should be understood
here as a $C^0$ function of $x$ that goes to $0$ with $\delta$ uniformly in $x$. \newline
Furthermore, by \eqref{eq:1storderTaylordiffeo}, for $\delta$ small enough,
\begin{align*}
\left|\dfrac{D_\delta^{-1}(x)-x}{\delta}\right|&\leq\dfrac{3}{2}|S'(x)|\\
\left|\dfrac{D_\delta^{-1}(x)-x-S'(x)}{\delta}\right|&\leq \dfrac{3}{4}
\end{align*}
uniformly in $x$. Similarly, by Lebesgue's differentiation theorem, for $\delta$ small enough and a.e $x\in\S^1$,
\begin{equation*}
\left|\dfrac{1}{D_\delta^{-1}(x)-x}\int_x^{D_\delta^{-1}x}(f^{\prime }(t)-f^{\prime }(x))dt\right|\leq \dfrac{1}{2}
\end{equation*}
Hence the quantity
considered is bounded almost everywhere by the $L^1$ function $\dfrac{3}{4}(|S|+|f^{\prime
}|)$, which is independent of $\delta$. \newline
Hence, Lebesgue's dominated convergence theorem apply, and one has 
\begin{equation*}
\dfrac{1}{\delta}\int_{\S ^1}|f\circ D_\delta^{-1}-f+\delta Sf^{\prime }|dm%
\underset{\delta\rightarrow 0}{\longrightarrow}0,
\end{equation*}
i.e $\|(II)\|_{L^1}=o(\delta)$. \newline
Let us assume now that if $D_\delta$ is a $C^{k+1}$ diffeomorphism, $%
\|(II)\|_{W^{k-1,1}}=o(\delta)$ holds for $f\in W^{k,1}(\S ^1)$. \newline
If $D_\delta$ is $C^{k+2}$, let $f\in W^{k+1,1}(\S ^1)$. We write as usual $%
\|(II)\|_{W^{k,1}}=\|(II)^{\prime }\|_{W^{k-1,1}}+\|(II)\|_{L^1}$, where 
\begin{align*}
(II)^{\prime }=(f\circ D_\delta^{-1})^{\prime }-f^{\prime }+(\delta
Sf^{\prime })^{\prime }=L_{D_\delta}(f^{\prime })-f^{\prime }-\delta
R(f^{\prime }).
\end{align*}
Hence, by induction hypothesis, $\|(II)^{\prime }\|_{W^{k-1,1}}=o(\delta)$
and the conclusion follows from the case $k=1$.
\end{proof}

We are left to verify that the transfer operator $L_{D_{\delta }}$ has a
second order Taylor's expansion at $\delta =0$. In that perspective, one
needs to obtain more precise information on the family of diffeomorphisms $%
(D_{\delta })_{\delta \in \lbrack 0,\bar{\delta}]}$. 

\begin{lemma}
Let $S\in C^{2}(\S ^1)$ and $D_\delta=Id+\delta S$ be a $C^{2}$
diffeomorphism. Then 
\begin{equation}  \label{eq:2ndorderTaylordiffeo}
\left\|\dfrac{D_\delta^{-1}-Id+\delta.S}{\delta^2}-S.S^{\prime
}\right\|_{C^0(\S ^1)}\underset{\delta\rightarrow 0}{\longrightarrow}0,
\end{equation}
which we sum up in $D_\delta^{-1}=Id-\delta.S+\delta^2S.S^{\prime}+o(
\delta^2)$.
\end{lemma}

\begin{remark}
The term $o(\delta^2)$ should be understood as a $C^0$ function that goes to 
$0$ with $\delta$, uniformly in $x$. 
\\Although we do not prove it here, in the general case $S\in C^{k+2}(\S^1)$, the Taylor expansion holds in $C^k(\S^1)$ topology, which means the $o(\delta)$ term should be understood as a $C^k(\S^1)$ function that goes to $0$ in $C^k$ norm as $\delta\to 0$.
\end{remark}
The proof is similar to the proof of \eqref{eq:1storderTaylordiffeo}.
\\Next, we introduce a \emph{second derivative operator}, $Q:W^{k,1}(\S ^1)\to
W^{k-2}(\S ^1)$, defined by 
\begin{equation}  \label{def:det2ndderivop2}
Qf:=(f.S^2)^{\prime \prime }.
\end{equation}
We are now in a position to formulate the following result:

\begin{proposition}
\label{prop:2ndorderTaylordetop} Let $k\geq 2$ and $D_{\delta }:\S %
^{1}\rightarrow \S ^{1}$ be a $C^{k+1}$ diffeomorphism (as in %
\eqref{def:diffeoperturb}), and let $L_{D_{\delta }}:W^{k,1}(\S %
^{1})\rightarrow W^{k,1}(\S ^{1})$ be its transfer operator, and $Q$, $R$ be
the derivatives operator \eqref{def:detderivop}, \eqref{def:det2ndderivop2}.
One has 
\begin{equation}
\left\Vert \dfrac{L_{D_{\delta }}-Id-\delta R}{\delta ^{2}}-\dfrac{1}{2}%
Q\right\Vert _{W^{k,1}\rightarrow W^{k-2,1}}\underset{\delta \rightarrow 0}{%
\longrightarrow }0.  \label{eq:2ndorderTaylordettransferop}
\end{equation}
\end{proposition}

\begin{proof}
One may start by remarking that 
\begin{align*}
L_{D_{\delta }}f-f-\delta R(f)-2^{-1}\delta ^{2}Q(f)& =J_{\delta }f\circ
D_{\delta }^{-1}-f+\delta (fS)^{\prime}-2^{-1}\delta ^{2}(fS^{2})^{\prime \prime
} \\
& =(I)+(II)
\end{align*}%
with 
\begin{align*}
(I):=& (J_\delta-1)f\circ D_\delta^{-1}+\delta fS'-\delta^2(f'S'S+f((S')^2+SS'')) \\
(II):=& f\circ D_{\delta }^{-1}-f+(\delta S-\delta ^{2}SS^{\prime
})f^{\prime}-2^{-1}\delta ^{2}S^{2}f^{\prime \prime }.
\end{align*}%
In view of \eqref{eq:weight} and \eqref{def:detderivop}, it is possible to
rewrite $(I)=-\delta \left( L_{D_{\delta }}(fS^{\prime })-fS^{\prime
}-\delta R(fS^{\prime })\right) $. Taking into account Proposition \ref%
{prop:1storderTaylordetop}, we get $\Vert (I)\Vert _{W^{k-2,1}}=o(\delta
^{2})$ as $\delta \rightarrow 0$. 
\\For the term $(II)$, we proceed
by induction on $k\geq 2$. \newline
For $k=2$, we start by evaluating $(II)$ at $x\in \S ^{1}$. Using mean value
theorem one gets 
\begin{align*}
& \int_{x}^{D_{\delta }^{-1}(x)}(f^{\prime }(t)-f^{\prime }(x))dt+\left(
D_{\delta }^{-1}(x)-x+\delta S(x)-\delta ^{2}S(x)S^{\prime }(x)\right)
f^{\prime}-2^{-1}\delta ^{2}S^{2}(x)f^{\prime \prime }(x) \\
& =\int_{x}^{D_{\delta }^{-1}(x)}\int_{x}^{t}(f^{\prime \prime
}(s)-f^{\prime \prime }(x))dsdt+\left( D_{\delta }^{-1}(x)-x+\delta
S(x)-\delta ^{2}S(x)S^{\prime }(x)\right) f^{\prime }(x) \\
& +2^{-1}\left( (D_{\delta }^{-1}(x)-x)^{2}-\delta ^{2}S^{2}(x)\right)
f^{\prime \prime }(x).
\end{align*}%
For a.e $x\in \S ^{1}$, this last quantity is $o(\delta ^{2})$. Indeed, by
virtue of \eqref{eq:1storderTaylordiffeo} one has $(D_{\delta
}^{-1}-Id)^{2}=\delta ^{2}S^{2}+o(\delta ^{2})$, hence 
\begin{equation}\label{eq:LDT}
\dfrac{1}{\delta ^{2}}\int_{x}^{D_{\delta }^{-1}(x)}\int_{x}^{t}(f^{\prime
\prime }(s)-f^{\prime \prime }(x))dsdt=\underset{\underset{\delta
\rightarrow 0}{\longrightarrow }S^{2}(x)}{\underbrace{\dfrac{(D_{\delta
}^{-1}(x)-x)^{2}}{\delta ^{2}}}}\underset{\underset{\delta \rightarrow 0}{%
\longrightarrow }0}{\underbrace{\dfrac{1}{(D_{\delta }^{-1}(x)-x)^{2}}%
\int_{x}^{D_{\delta }^{-1}(x)}\int_{x}^{t}(f^{\prime \prime }(s)-f^{\prime
\prime }(x))dsdt}},
\end{equation}%
by Lebesgue's differentiation theorem, and 
\begin{equation*}
2^{-1}\left((D_{\delta }^{-1}(x)-x)^{2}-\delta^{2}S^{2}(x)\right)f^{\prime \prime }(x)=o(\delta^2)f^{\prime\prime}(x).
\end{equation*}%
Similarly, by \eqref{eq:2ndorderTaylordiffeo} 
\begin{equation*}
\left(D_{\delta }^{-1}(x)-x+\delta S(x)-\delta ^{2}S(x)S^{\prime
}(x)\right)f^{\prime }(x)=o(\delta^2)f^\prime(x)
\end{equation*}%
By virtue of \eqref{eq:1storderTaylordiffeo} and \eqref{eq:2ndorderTaylordiffeo}, one has
\begin{align*}
\dfrac{1}{\delta^2}\left|2^{-1}\left((D_{\delta }^{-1}(x)-x)^{2}-\delta^{2}S^{2}(x)\right)f^{\prime \prime }(x)\right|&\leq \dfrac{3}{4}|f^{\prime\prime}(x)|\\
\dfrac{1}{\delta^2}\left|\left(D_{\delta }^{-1}(x)-x+\delta S(x)-\delta ^{2}S(x)S^{\prime
}(x)\right)f^{\prime }(x)\right|&\leq \dfrac{3}{4}|f^\prime(x)|,
\end{align*}
uniformly in $x\in\S^1$ for $\delta$ small enough. Finally, for $\delta$ small enough, the right-hand side of \eqref{eq:LDT} is bounded almost everywhere by $\dfrac{3}{4}S^2(x)$.
\\From what precedes, if $\delta $ is small enough, $%
(II)$ is bounded a.e by the $L^{1}$ function $\dfrac{3}{4}\left(S^{2}+|f^{\prime}|+|f^{\prime \prime }|\right) $, independent of $\delta$.
\\Thus, by Lebesgue's
dominated convergence theorem, $\Vert (II)\Vert _{L^{1}}=o(\delta ^{2})$. 
\newline
Assuming that the property holds at $k\in \mathbb{N}$, we now take $%
D_{\delta }$ to be a $C^{k+2}$ diffeomorphism, and let $f\in W^{k+1,1}(\S %
^{1})$. We have 
\begin{equation*}
\Vert (II)\Vert _{W^{k-1,1}}=\Vert (II)^{\prime }\Vert _{W^{k-2,1}}+\Vert
(II)\Vert _{L^{1}},
\end{equation*}%
and computing yields 
\begin{equation*}
(II)^{\prime }=L_{D_{\delta }}(f^{\prime })-f^\prime-\delta R(f^{\prime})-2^{-1}\delta
^{2}Q(f^{\prime }).
\end{equation*}%
Thus, by induction hypothesis, $\Vert (II)^{\prime }\Vert
_{W^{k-2,1}}=o(\delta ^{2})$, and the conclusion follows from the case $k=2$.
\end{proof}

\subsection{Application: Explicit perturbation of the doubling map}

An example of system to which the previous discussion apply is the following
perturbation of the doubling map $(T_{\delta })_{\delta \in \lbrack 0,%
\overline{\delta }]}$ defined by 
\begin{equation}
T_{\delta }(x):=2x+\delta \sin (4\pi x)~\mod 1  \label{eq:classicalex}
\end{equation}%
which falls under the setup described in Section \ref{sec:pertofexpmaps}
with $T_{0}(x):=2x\mod 1$ and $D_{\delta }(x):=x+\delta \sin (2\pi x)\mod 1$%
, and the spaces $B_{ss}=W^{4,1}(\S ^{1})\subset B_{s}=W^{3,1}(\S %
^{1})\subset B_{w}=W^{2,1}(\S^1)\subset W^{1,1}(\S ^{1})=B_{ww}$. \newline
Indeed, the system satisfies uniform Lasota-Yorke estimates (for $%
\delta _{0}$ small enough) by Lemma \ref{Lemsu} and Proposition \ref%
{propora}; in particular Theorem \ref{gap} apply.
Furthermore, this example satisfies the regularity requirements of Section \ref%
{sec:pertofexpmaps}, so that Propositions \ref{prop:strongcontinuity}, \ref%
{prop:1storderTaylordetop} and \ref{prop:2ndorderTaylordetop} apply. 
\\This implies that the assumptions of Theorems \ref%
{thm:linresp} and \ref{thm:quadresp} are satisfied, so for this family of
systems, linear response holds if one considers the invariant density $%
h_{\delta }$ as a $W^{2,1}(\S ^{1})$ function, and quadratic response holds
if one considers the invariant density $h_{\delta }$ as a $W^{1,1}(\S ^{1})$
function. \newline
This example is certainly well-known, and may be obtained by other
methods (\cite{Li2} for linear response or \cite{GL,JS} for higher-order
response). However, the nice feature of this example is the possibility to
compute everything: here we have 
\begin{align*}
L_{0}f(x)& =\dfrac{1}{2}\left[ f\left( \dfrac{x}{2}\right) +f\left( \dfrac{1+x}{%
2}\right) \right] \\
\dot{L}h_{0}(x)& =R[L_{0}h_{0}](x)=-2\pi \cos (2\pi x) \\
\ddot{L}h_{0}(x)& =Q[L_{0}h_{0}](x)=8\pi ^{2}\cos (4\pi x)
\end{align*}%
with $h_{0}=1$ the invariant density of the unperturbed system.
Notice that for $f(x)=\cos(2\pi x)$ we have
\[L_0f(x)=\dfrac{1}{2}\left[\cos(\pi x)+\cos(\pi x+\pi)\right]=0. \]
Hence, applying %
\eqref{linresp} yields 
\begin{align*}
\dfrac{d}{d\delta }h_\delta(x)|_{\delta=0}=\sum_{n=0}^{\infty }L_{0}^{n}\dot{L}%
h_{0}&=\sum_{n=0}^{\infty }L_{0}^{n}(-2\pi \cos (2\pi x))\\
&=-2\pi \cos (2\pi x)-2\pi\sum_{n\geq 1}\underset{=0}{\underbrace{L_0^n(\cos(2\pi x))}}\\
&=-2\pi \cos (2\pi x)
\end{align*}%
Similarly, we may compute the quadratic term with \eqref{eq:quadresp}: 
\begin{equation*}
\dfrac{d^2}{d\delta^2}h_\delta(x)|_{\delta=0}=(R(1,L_{0})\dot{L}%
)^{2}h_{0}+R(1,L_{0})Qh_{0}=Qh_{0}+L_{0}Qh_{0}=8\pi ^{2}(\cos (4\pi x)+\cos
(2\pi x))
\end{equation*}%
so that one may obtain the following explicit, order two Taylor's expansion
for $h_{\delta }$: 
\begin{equation}
h_{\delta }(x)=1-2\pi \delta \cos (2\pi x)+4\pi ^{2}\delta ^{2}\left( \cos
(4\pi x)+\cos (2\pi x)\right) +o(\delta ^{2}).
\end{equation}

\noindent \textbf{Acknowledgments.} S.G. is partially supported by the
research project PRIN 2017S35EHN\_004 "Regular and stochastic behavior in
dynamical systems" of the Italian Ministry of Education and Research. 
\newline
J.S is supported by the European Research Council (ERC) under the European
Union's Horizon 2020 research and innovation program (grant agreement No
787304).

\end{document}